\documentclass[twoside,reqno]{amsart}
\usepackage{amsfonts,amsmath,amscd,amsthm,amssymb}
\usepackage{diagbox,booktabs, adjustbox}
\usepackage{mathtools}
\usepackage[mathscr]{euscript} 
\usepackage{graphics}
\usepackage{url}
\usepackage{wrapfig}
\usepackage{lscape}
\usepackage{rotating}
\usepackage{epsfig}
\usepackage{cite}
\usepackage{color}
\usepackage{booktabs}
\usepackage{multirow}
\usepackage{lipsum,multicol}

\usepackage{tikz}
\usetikzlibrary{patterns}

\usepackage{amssymb}
\usepackage{bbm}
\usepackage{multirow}

\newtheorem{theorem}{Theorem}[section]

\theoremstyle{definition}

\theoremstyle{remark}
\newtheorem{remark}[theorem]{Remark}

\numberwithin{equation}{section}
\numberwithin{table}{section}

\def\jump#1{[\![{#1}]\!]} 
\def\mean#1{\{\!\!\{{#1}\}\!\!\}} 

\def\intc#1{\int_{I_j}{#1}\mathrm{dx} } 
\newcommand{\bld}[1]{\boldsymbol{#1}}

\begin{document}

\title[Dispersion for energy conserving  DG]{Dispersive behavior 
of an energy-conserving discontinuous Galerkin method for the one-way wave equation}

\author{Mark Ainsworth}
\address{Division of Applied Mathematics, Brown University, 182 George St,
Providence RI 02912, USA.}
\email{Mark\_Ainsworth@brown.edu}

\author{Guosheng Fu}
\address{Division of
Applied Mathematics, Brown University, Providence, RI
 02912}
\curraddr{}
\email{guosheng\_fu@brown.edu}



\dedicatory{}


\begin{abstract}
The dispersive behavior of the
recently proposed energy-conserving discontinuous Galerkin (DG) method
by Fu and Shu \cite{FuShu18} is analyzed and compared with the classical centered and 
upwinding DG schemes. It is shown that the new scheme gives a 
significant improvement over the 
classical centered and upwinding DG schemes in terms of dispersion error.
Numerical results are presented to support the theoretical findings.
\end{abstract}

\maketitle
\section{Introduction}
\label{sec:intro}
The quest for stable
and accurate
schemes for systems of hyperbolic conservation laws
has occupied researchers for several decades and continues to this day 
\cite{Abgrall16, Abgrall17}
with active research into 
finite difference methods, finite volume methods, spectral methods and 
a variety of finite element Galerkin schemes.
The current consensus seems to be that discontinuous Galerkin (DG) schemes \cite{Cockburn00} are the most 
promising, although they too have their drawbacks even if one restricts attention to linear hyperbolic systems.
In this setting, one wishes to have numerical schemes which are able to propagate 
discrete waves at, or near to,  the same
speed at which continuous waves are propagated 
by the original hyperbolic system. The dispersive and dissipative behavior of a numerical 
scheme compared with that of the original system is of considerable interest and had been widely studied
\cite{Ihlenburg98,Durran99,Cohen02,Ainsworth04,AinsworthMonk06, Ainsworth14b}.


This paper is devoted to a dispersion analysis of the recently proposed energy-conserving DG method \cite{FuShu18}.
To fix ideas, we consider the following one-way wave equation with unit wave speed:
\begin{align}
\label{advection1d}
 u_t +  u_x & =0, &&\hspace{-2.8cm}  x\in \mathbb{R}, t >0,
\end{align}
for suitable initial data.
To begin with, we confine our attention to uniform partitions of $\mathbb{R}$ consisting of 
cells of size $h>0$, whose nodes are located at the points $h(\mathbb{Z}+1/2)$.
Denote the $j$th cell $I_j = ((j-1/2) h, (j+1/2)h)$, 
and let $V_h^N$ denote the space of piecewise continuous polynomials of degree $N$
on the partition:
\begin{align}
\label{space-1d}
 V_h^{N} = \left\{v \in L^2(\mathbb{R}): v|_{I_j} \in \mathbb{P}_N(I_j),
~~ \forall j\in\mathbb{Z}\right\}, 
\end{align}
where $\mathbb{P}_{N}(I_j)$ denotes the set of polynomials of degree up to $N\ge 0$ 
defined on the cell $I_j$.
For any function $p\in V_h^N$, we let $p_{j-1/2}^-$ and $p_{j-1/2}^+$ be the values
of $p$ at the node  $x_{j-1/2} = (j-1/2)h$, from the left
cell, $I_{j-1}$, and from the right cell, $I_{j}$, respectively.
In what follows, we employ $\jump p|_{j-1/2} = p^+_{j-1/2} - p^-_{j-1/2}$ and $\mean p|_{j-1/2} =
\frac12(p^+_{j-1/2} +
p^-_{j-1/2})$ to represent the jump and the mean value of $p$ at each node. 

The DG method for \eqref{advection1d} reads as follows: 
Find the unique function $u_h = u_h(t)\in V_h^N$ such that
\begin{align}
\label{scheme:dg}
\intc{(u_h)_tv_h} - \intc{u_h (v_h)_x} +
\widehat{u}_h v_h^-|_{j+\frac12} -
\widehat{u}_h v_h^+|_{j-\frac12} = &\;0,
\end{align}
holds for all $v_h\in V_h^N$ and all $j\in\mathbb{Z}$.
The classical upwinding DG method, denoted by (U), uses numerical fluxes chosen to be
\[
  \widehat{u}_h|_{j-\frac12} =  \; \mean {u_h}|_{j-\frac12} +\frac12\jump {u_h}|_{j-\frac12},
\]
while the centered DG method, denoted by (C),  uses numerical fluxes given by
\[
  \widehat{u}_h|_{j-\frac12} =  \; \mean {u_h}|_{j-\frac12}.
\]

The  method (U) is energy dissipative in the sense that 
\begin{align}
\label{disp}
 \frac12\frac{\mathrm{d}}{\mathrm{dt}}\int_{\mathbb{R}}u_h^2\mathrm{dx}
 =-\sum_{j\in\mathbb{Z}}\frac12(\jump{u_h})^2|_{j-1/2}\le 0,
\end{align}
while the method (C) is energy-conservative 
\[
 \frac12\frac{\mathrm{d}}{\mathrm{dt}}\int_{\mathbb{R}}u_h^2\mathrm{dx}
 =0.
\]
Despite being energy conserving, the centered flux scheme (C) is seldom used in practice owing
to the reduced stability properties of the scheme compared with the upwinding scheme (U), c.f. \cite{CockburnShu98}. 
For this reason, the scheme (U) is often preferred  and the lack of energy conservation
tolerated. Expression \eqref{disp} shows that if the jump terms $\jump{u_h}|_{j-\frac12}$
are non-zero then energy will be dissipated and, importantly, 
that there is no mechanism whereby the dissipated
energy can be regained by the scheme.

Recently, Fu and Shu \cite{FuShu18} proposed
an energy-conserving discontinuous Galerkin (DG) method for linear symmetric hyperbolic systems, and 
gave an optimal a priori error estimate
for the method in one dimension, and in multi-dimensions on tensor-product meshes. 
Numerical evidence presented in \cite{FuShu18} suggests that the scheme is 
optimally convergent on general triangular meshes, and 
has
superior dispersive properties of the new DG method comparing with the 
(energy dissipative) upwinding DG method (U) and the (energy conservative) 
centered DG method (C) translating into improved accuracy
for long time simulations.

The method of Fu and Shu is unusual in that it begins at the continuous level
by introducing an auxiliary advection equation (with the opposite wave speed to that in 
the equation for $u$), to obtain
the following (decoupled) system:
 \begin{subequations}
 \label{aux-adv}
\begin{align}
 u_t +  u_x & =0, &&\hspace{-2.8cm}  x\in\mathbb{R}, t>0,\\
 \phi_t -  \phi_x & =0, &&\hspace{-2.8cm}  x\in\mathbb{R}, t>0,
\end{align}
with initial condition $u(x,0)=u_0(x)$ and $\phi(x,0)=0$.
\end{subequations}
Obviously the solution $\phi$ is identically {\it zero}. However, this will not be the case 
for the DG approximation \cite{FuShu18} of the system, where the (non-zero) 
approximation of the second
equation is exploited to obtain energy conservation at the discrete level.

The DG method \cite{FuShu18} 
for \eqref{aux-adv} reads as follows: 
Find the unique function $(u_h, \phi_h) = (u_h(t), \phi_h(t))
\in V_h^N\times V_h^N$ such that
 \begin{subequations}
 \label{scheme:adv1d}
\begin{align} \label{scheme:adv1d-1}
\intc{(u_h)_tv_h} - \intc{u_h (v_h)_x} +
\widehat{u}_h v_h^-|_{j+\frac12} -
\widehat{u}_h v_h^+|_{j-\frac12} = &\;0,\\
\label{scheme:adv1d-2}
\intc{(\phi_h)_t\psi_h} + \intc{\phi_h (\psi_h)_x} -
\widehat{\phi}_h \psi_h^-|_{j+\frac12} +
\widehat{\phi}_h \psi_h^+|_{j-\frac12} =&\; 0,
\end{align}
 \end{subequations}
holds for all $(v_h, \psi_h)\in V_h^N\times V_h^N$ and all $j\in\mathbb{Z}$, where
$\widehat u_h$ and $\widehat\phi_h$ denote the numerical fluxes
\begin{subequations}
\label{flux:2x2}
\begin{align}
\label{flux:1}
 \widehat{u}_h|_{j-\frac12} =  &\; \mean {u_h}|_{j-\frac12} +\frac12\alpha\jump {\phi_h}|_{j-\frac12},\\ 
\label{flux:2}
 \widehat{\phi}_h|_{j-\frac12} = &\; \mean {\phi_h}|_{j-\frac12} +\frac12\alpha\jump {u_h}|_{j-\frac12}.
\end{align}
\end{subequations}
The constant in the numerical fluxes \eqref{flux:2x2} 
is chosen to be $\alpha=1$ in \cite{FuShu18}, and we denote the corresponding DG method by (A).
However, in this article, we will also consider the following choice
\begin{align}
 \label{alpha-opt}
 \alpha = \left\{
 \begin{tabular}{ll}
  $\sqrt{\frac{4}{3}}$ & if $N=0$,\\[1ex]
  $\sqrt{\frac{N(2N+3)}{(N+1)(2N+1)}}$ & if $N$ is odd,\\[1ex]
  $\sqrt{\frac{(N+1)(2N+1)}{N(2N+3)}}$ & if $N>0$ is even,
 \end{tabular}
 \right.
\end{align}
and denote the corresponding DG method by (A*).

Each of methods (A) and (A*) are energy conservative \cite{FuShu18} with respect to the following modified energy:
\begin{align}
\label{aux-e}
 \frac12\frac{\mathrm{d}}{\mathrm{dt}}\int_{\mathbb{R}}(u_h^2+\phi_h^2)\mathrm{dx}
 =0. 
\end{align}
Of course, this does not mean that the individual energy 
$\int_{\mathbb{R}}u_h^2\mathrm{dx}$ 
and $\int_{\mathbb{R}}\phi_h^2\mathrm{dx}$ are conserved in isolation. One way to view the scheme 
\eqref{scheme:adv1d} is to regard the auxiliary variable $\phi_h$ as a 
temporary store for collecting energy 
dissipated in \eqref{scheme:adv1d-1} which is then reinjected 
back into the equation for $u_h$
through the flux term $\jump{\phi_h}_{j-\frac12}$ in 
\eqref{flux:1}, still resulting in the overall energy of the system being conserved as shown by \eqref{aux-e}.
More interesting is that this exchange of energy in $u_h$ and $\phi_h$ also seems to 
render methods (A) and (A*)
superior to the method (U) and (C) in terms of numerical dispersion, as we shall see in Section \ref{sec:num}.

Section \ref{sec:main} contains a summary of the main results from our dispersion 
analysis in Section \ref{sec:disp}, and an explanation of the numerical results 
on uniform meshes conducted in Section \ref{sec:num}.
Conclusions are drawn in Section \ref{sec:conclude}.

\section{Illustration of Dispersive Behavior of the DG schemes}
\label{sec:num}
In this section  we carry out a simple numerical comparison of the above mentioned four DG methods.
We consider equation \eqref{aux-adv} on the unit interval $I=[0,1]$ with periodic boundary conditions,
and take the initial condition $u_0(x)=\sin(\omega x)$ with frequency $\omega = 2\pi$. 
Hence the true solution is $u(x,t) = \sin(2\pi(x-t)).$
Since we are primarily interested in the spatial discretisation, 
we use a sufficiently high-order time discretization so as to render the temporal error 
 negligible compared with the spatial error.

Numerical results for the four DG methods mentioned above with polynomial degree $N=0$
on 20 uniform cells at time $T=20$, 
with polynomial degree $N=1$
on 10 uniform cells at time $T=200$, 
and with polynomial degree $N=2$
on 4 uniform cells at time $T=300$ are presented in 
Fig.~\ref{fig:ex1-p0}--\ref{fig:ex1-p2}, respectively.
One observes from these figures that
the dissipative behavior of method (U), whilst the method (C) exhibits
large phase error compared with method (A), which in turn is inferior to method (A*).

\begin{figure}[ht!]
 \caption{Numerical solution $u_h$ at time  $T=20$.
Solid line: numerical solution. Dashed line: exact solution.
$N=0$, 20 uniform cells.
}
 \label{fig:ex1-p0}
 \vspace{1ex}
 \includegraphics[width=.45\textwidth,height=.13\textheight]{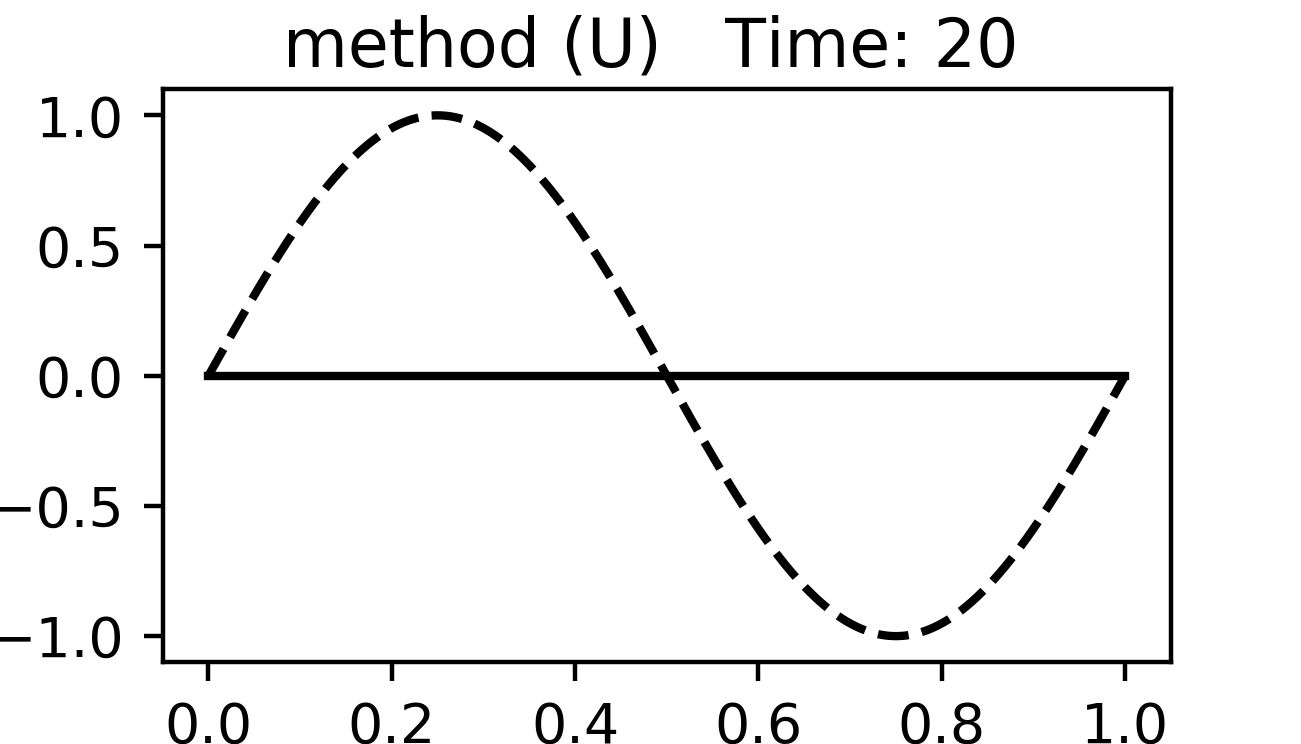}
 \includegraphics[width=.45\textwidth,height=.13\textheight]{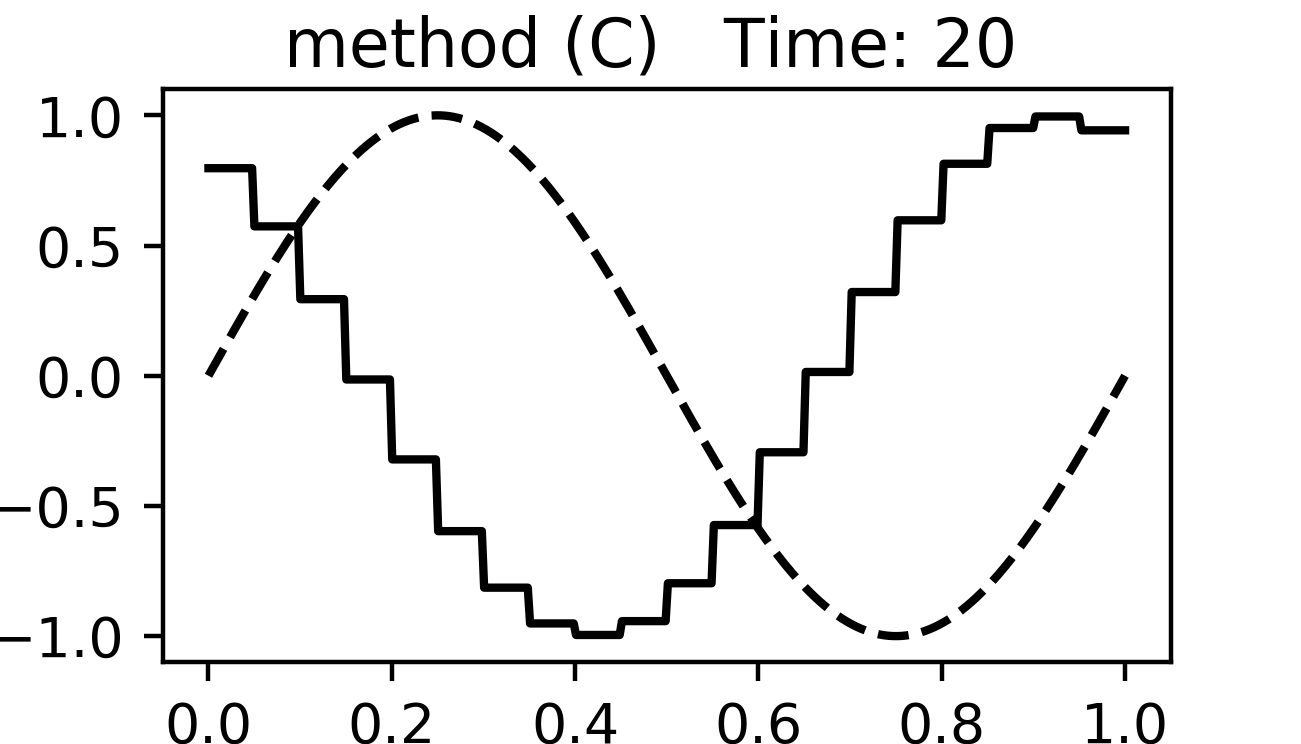}\\[1ex]
 \includegraphics[width=.45\textwidth,height=.13\textheight]{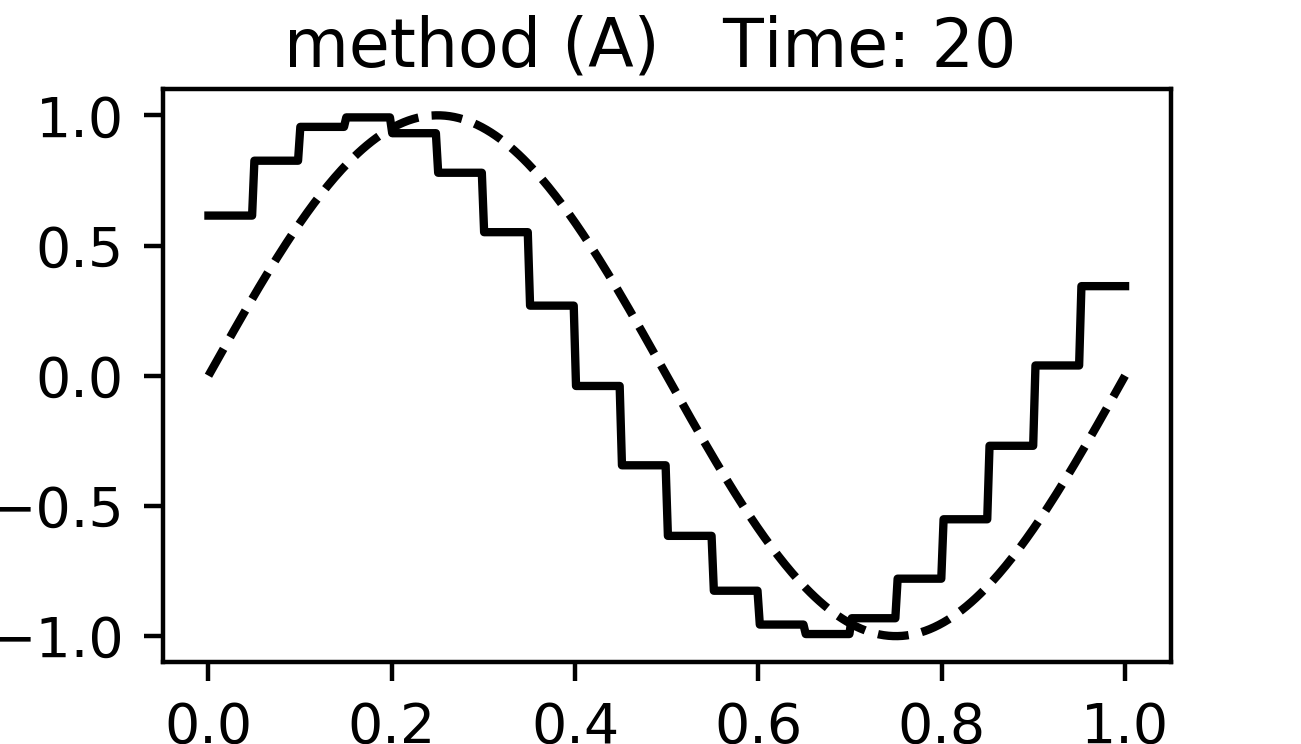}
 \includegraphics[width=.45\textwidth,height=.13\textheight]{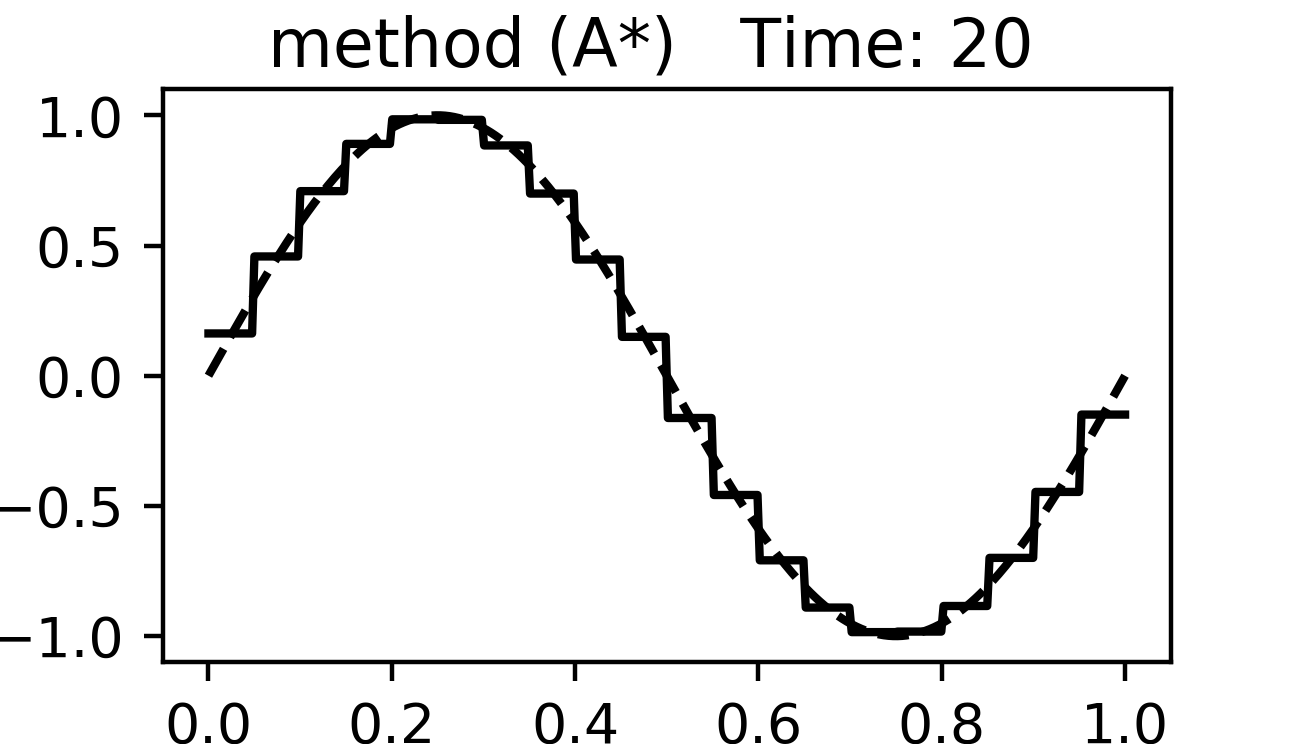} 
\end{figure}

\begin{figure}[ht!]
  \caption{Numerical solution $u_h$ at time $T=200$.
Solid line: numerical solution. Dashed line: exact solution.
$N=1$, 10 uniform cells.
}
 \label{fig:ex1-p1}
 \vspace{1ex}
 \includegraphics[width=.45\textwidth,height=.13\textheight]{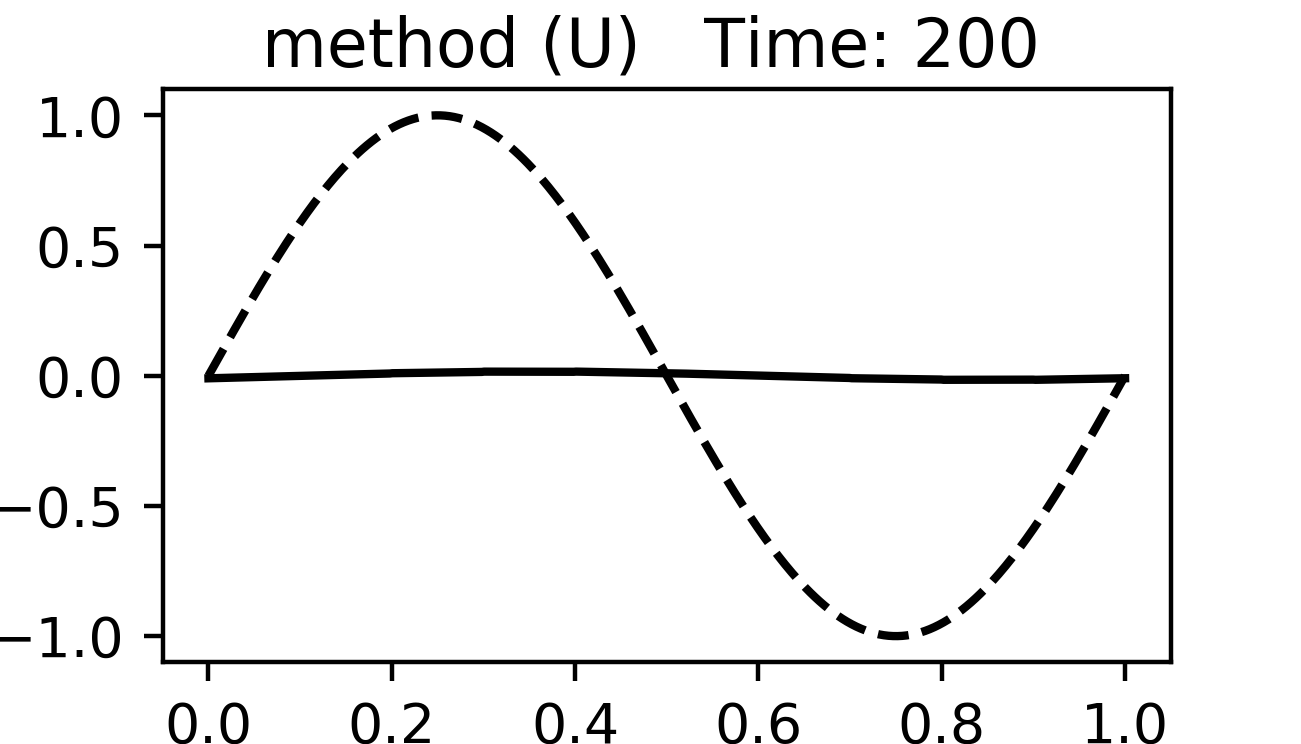}
 \includegraphics[width=.45\textwidth,height=.13\textheight]{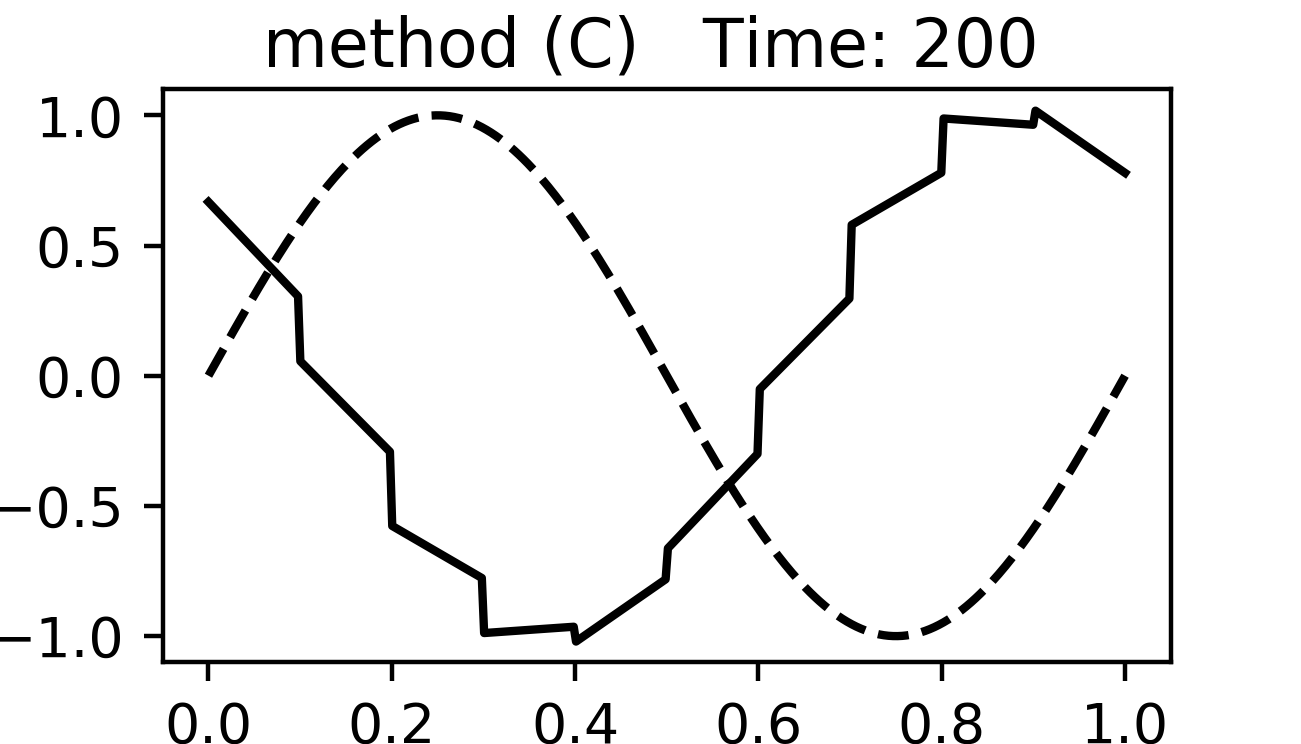}\\[1ex]
 \includegraphics[width=.45\textwidth,height=.13\textheight]{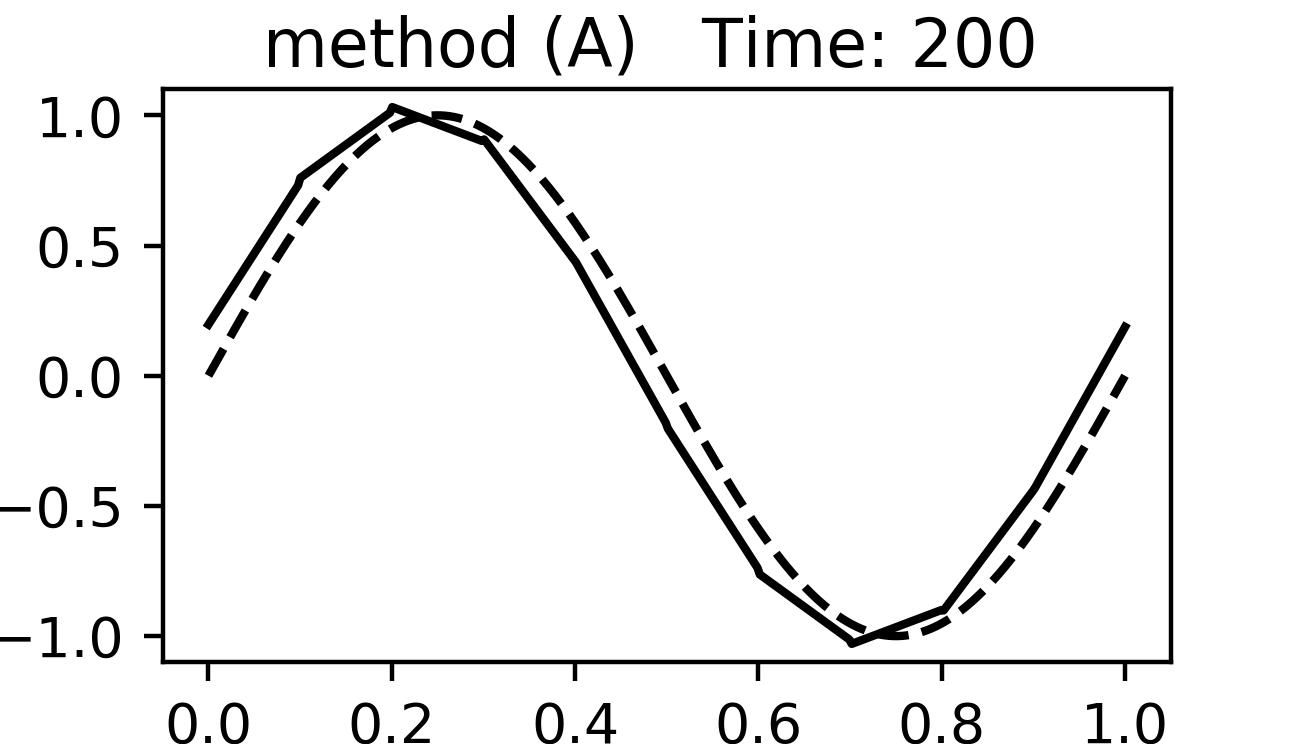}
 \includegraphics[width=.45\textwidth,height=.13\textheight]{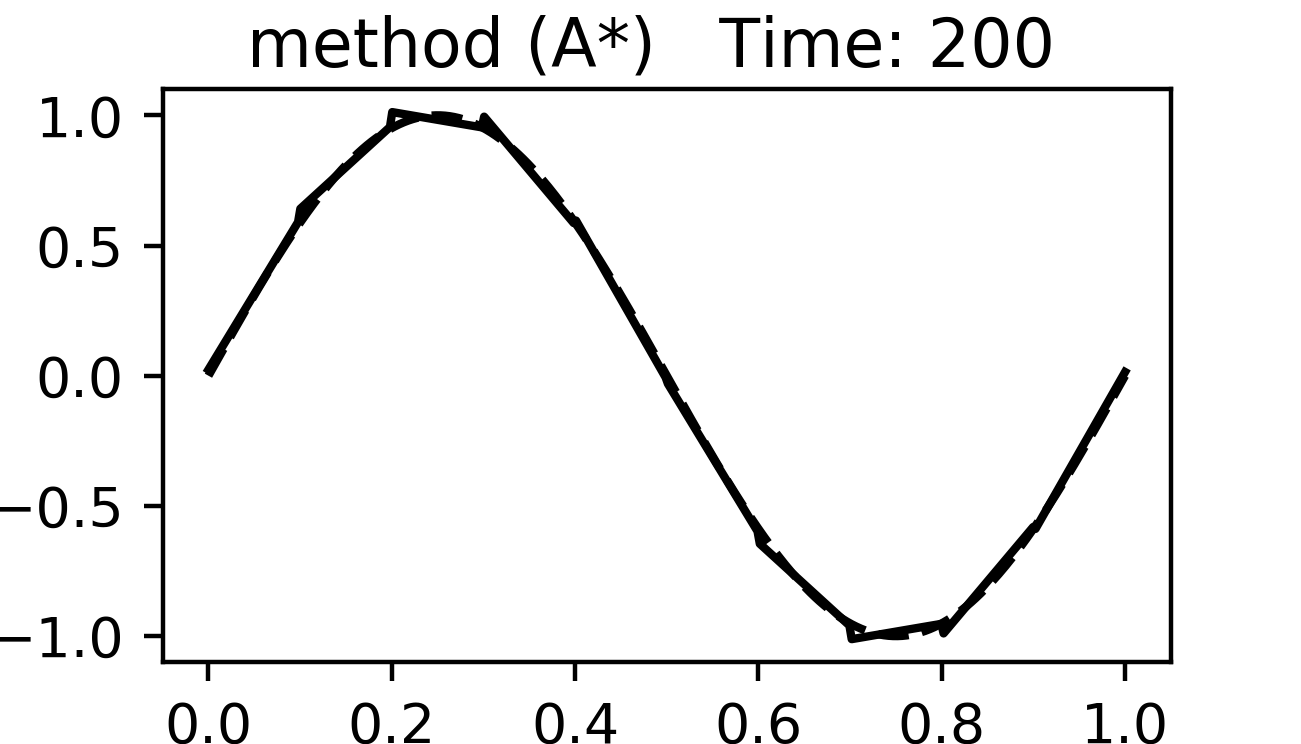} 
\end{figure}

\begin{figure}[ht!]
 \caption{Numerical solution $u_h$ at time  $T=300$.
Solid line: numerical solution. Dashed line: exact solution.
$N=2$, 4 uniform cells.
}
 \label{fig:ex1-p2}
 \vspace{1ex}
 \includegraphics[width=.45\textwidth,height=.13\textheight]{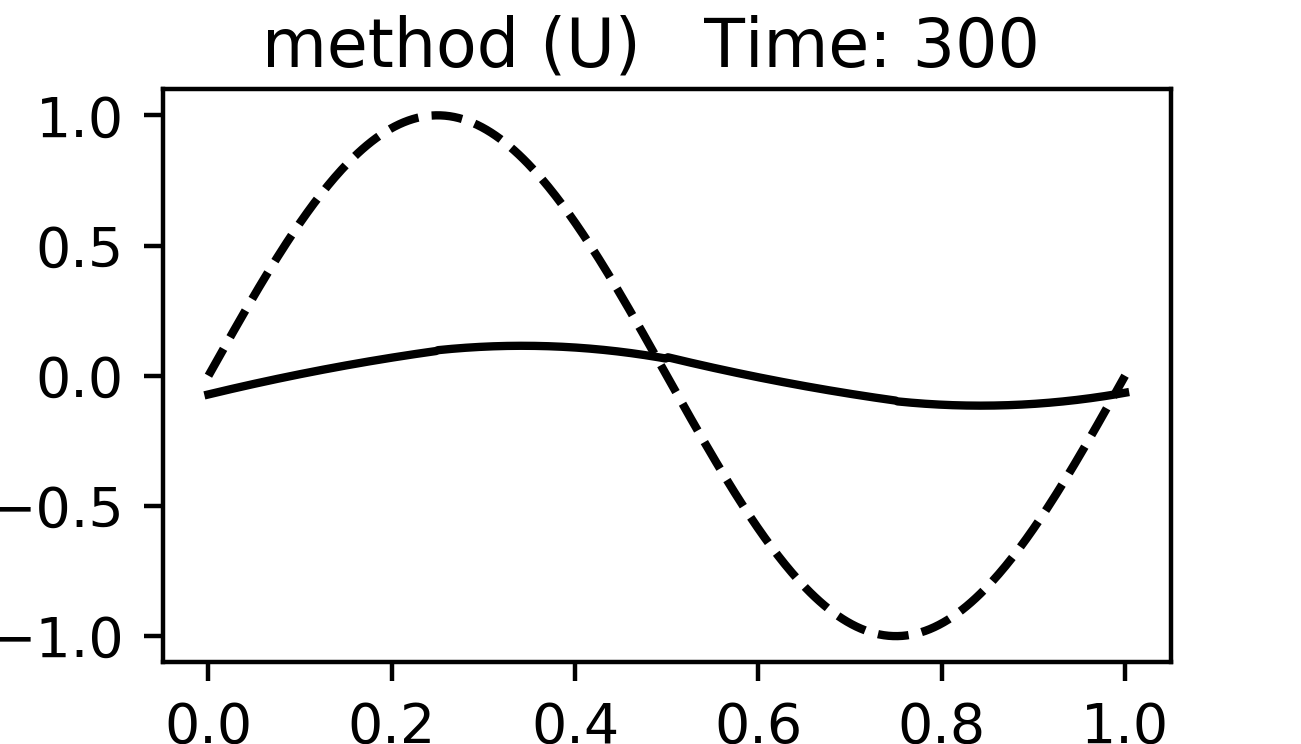}
 \includegraphics[width=.45\textwidth,height=.13\textheight]{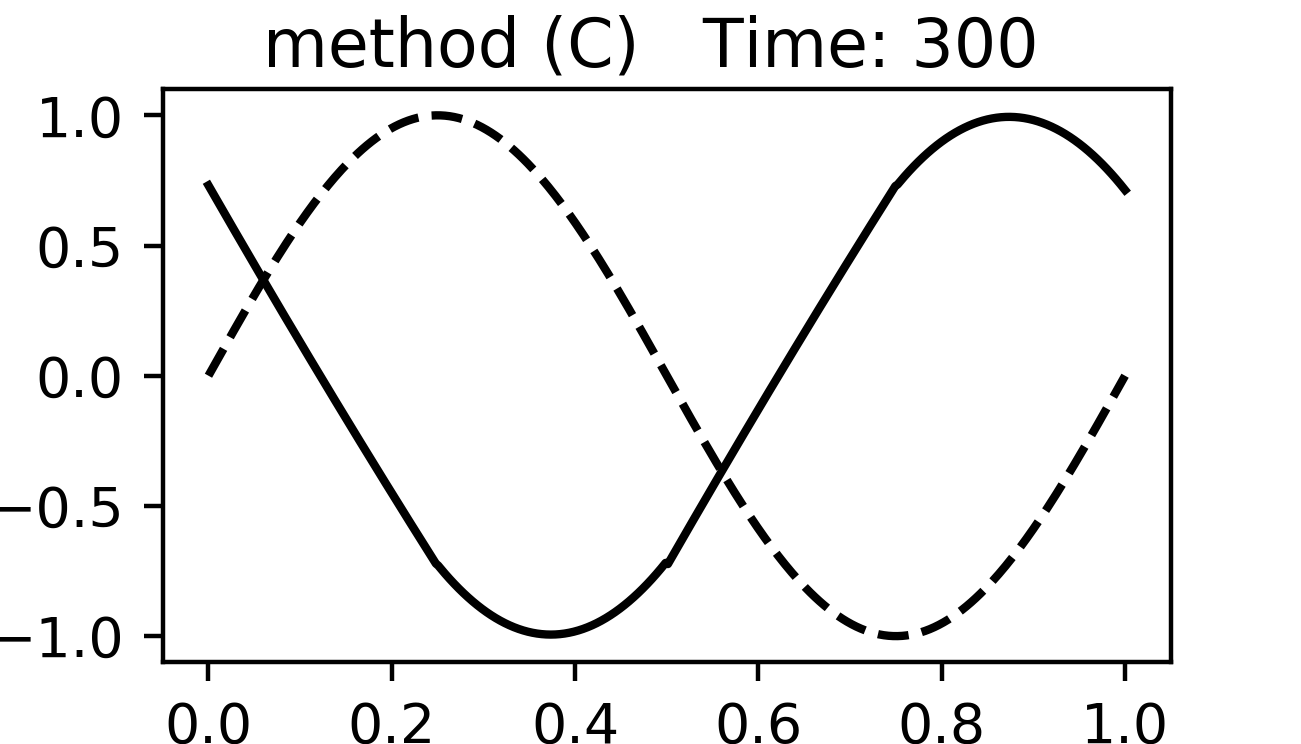}\\[1ex]
 \includegraphics[width=.45\textwidth,height=.13\textheight]{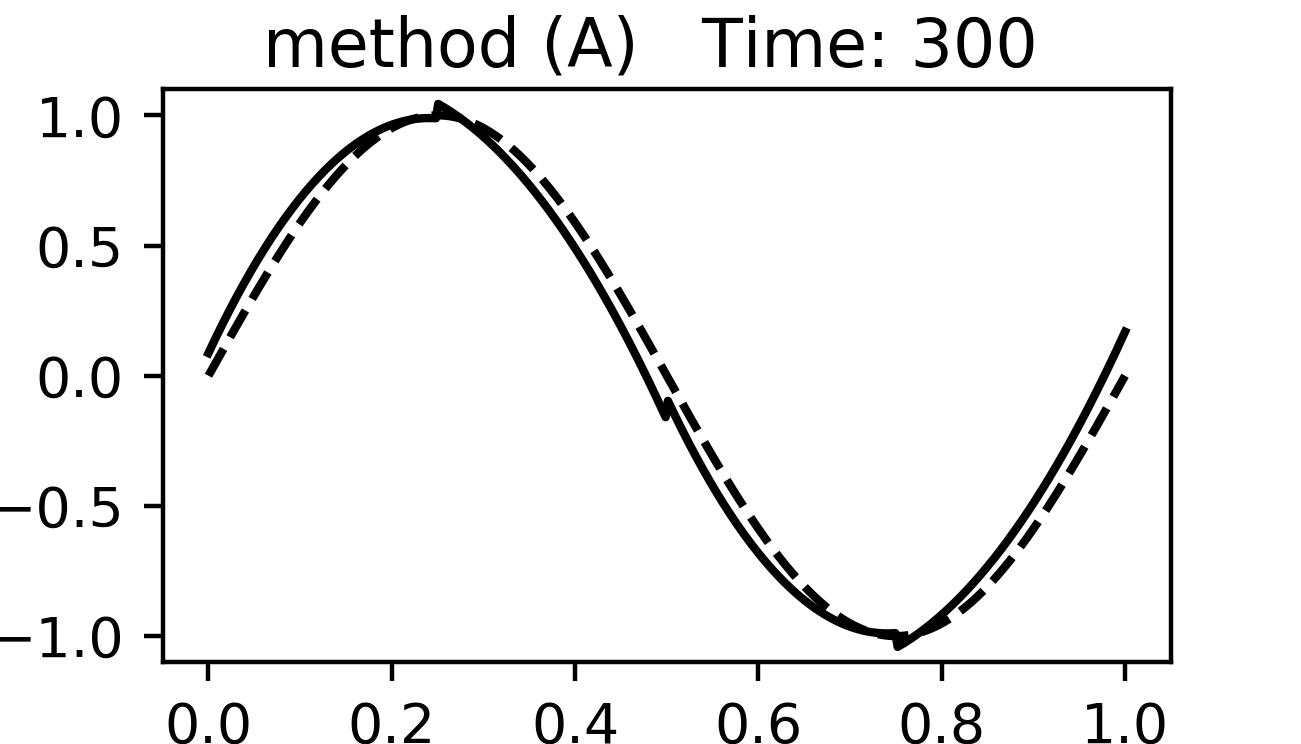}
 \includegraphics[width=.45\textwidth,height=.13\textheight]{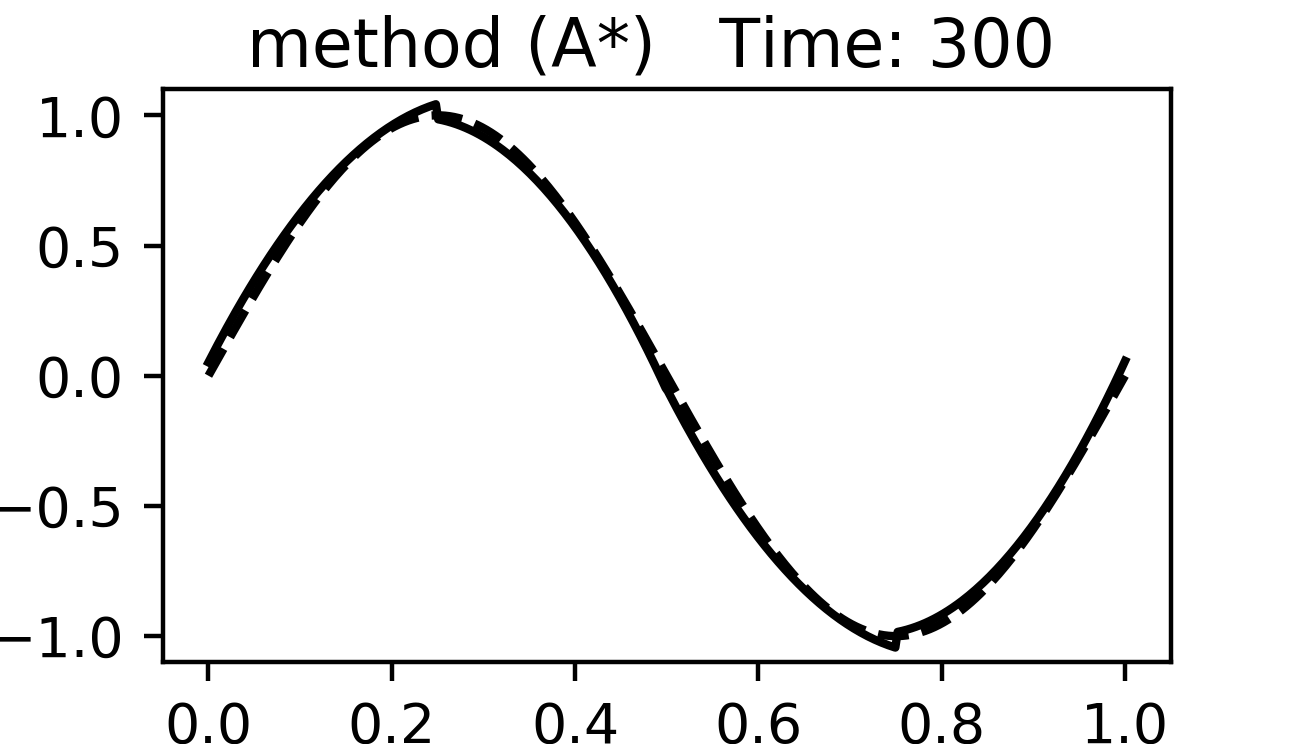} 
\end{figure}

For the $N=0$ case, we also compare the numerical approximations obtained at {\it different} times 
for the four methods
in Fig.~\ref{fig:ex1-diff}. 
It is striking that method (A*) at time $T=1500$ enjoys a similar 
accuracy to that of method (C) at time $T=5$ and method (A) at time $T=20$.
In Section \ref{sec:main} we will give a theoretical explanation for these observations.
\begin{figure}[ht!]
 \caption{Numerical solution $u_h$ at {\it different}  times.
Solid line: numerical solution. Dashed line: exact solution.
$N=0$, 20 uniform cells.
}
 \label{fig:ex1-diff}
 \vspace{1ex}
  \includegraphics[width=.45\textwidth,height=.13\textheight]{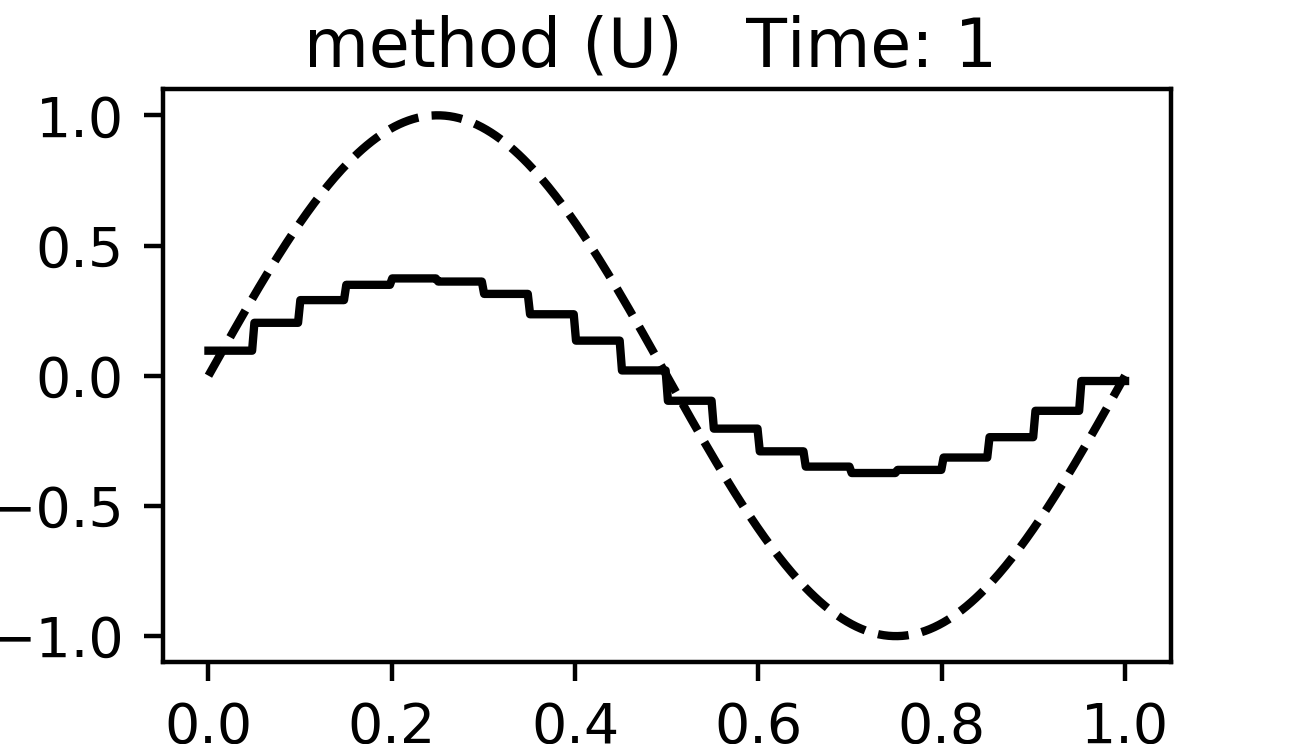}
 \includegraphics[width=.45\textwidth,height=.13\textheight]{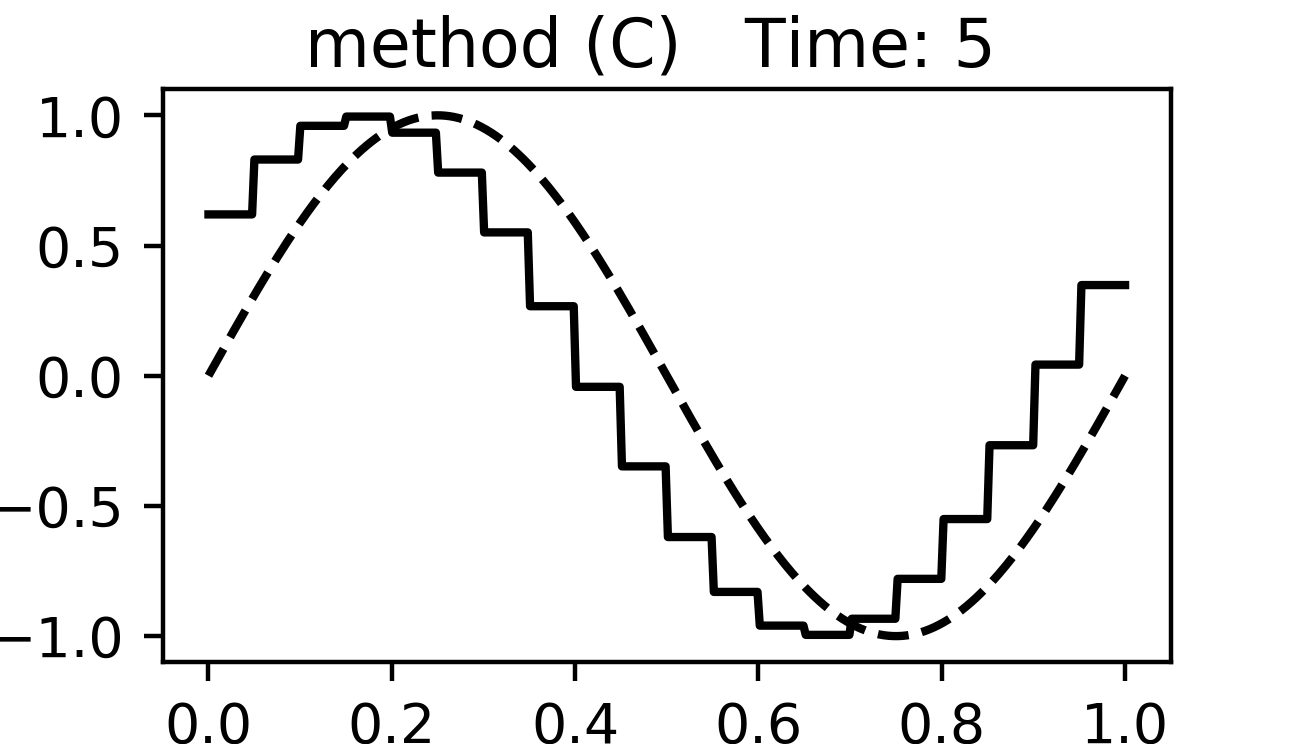}\\[1ex]
 \includegraphics[width=.45\textwidth,height=.13\textheight]{pwTrueauxp020.png}
 \includegraphics[width=.45\textwidth,height=.13\textheight]{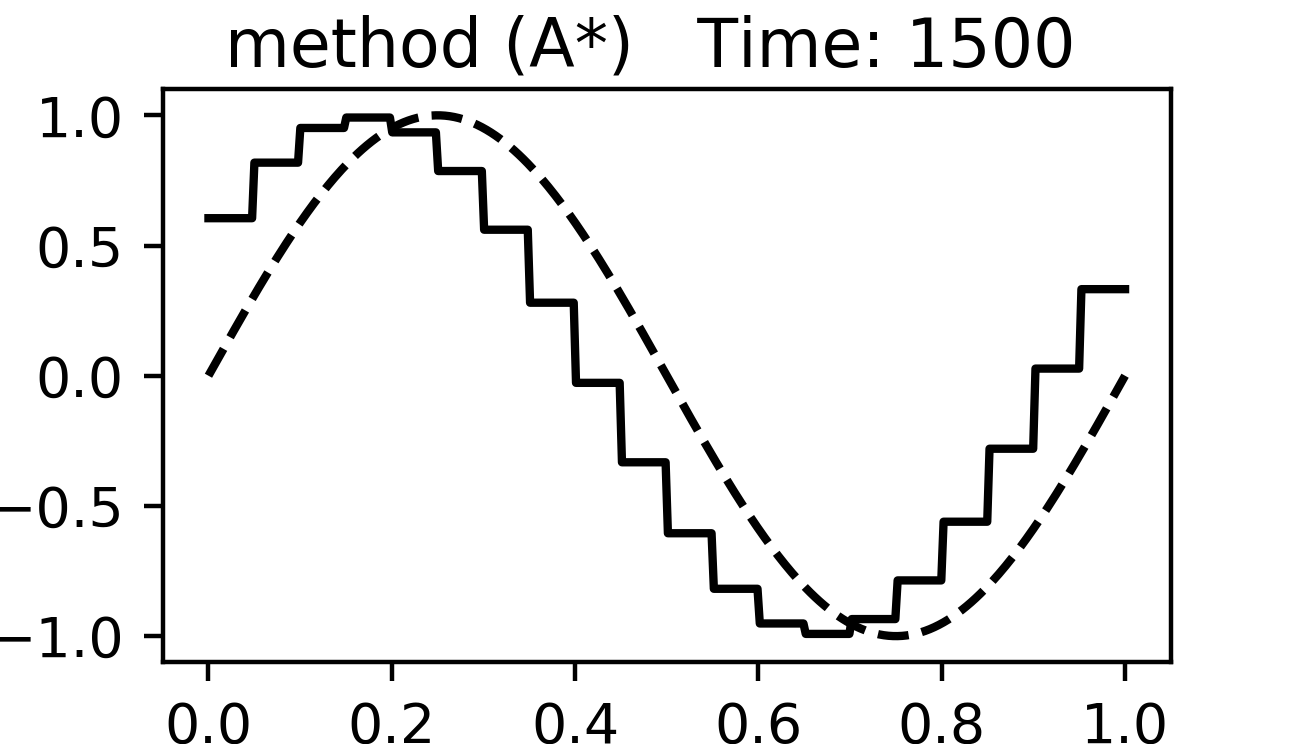}
\end{figure}

We also compare the numerical approximations obtained at time $T=40$ using
methods (A) and (A*) 
for $N=0$ on uniform and non-uniform meshes consisting of 20 cells in Fig.~\ref{fig:ex1-non1} and 
Fig.~\ref{fig:ex1-non2}, respectively. 
The non-uniform mesh is obtained by applying a uniformly distributed 
10\% random perturbation of the nodes in an 
uniform mesh. 
Comparing the results on the uniform mesh with the corresponding results on the non-uniform mesh, 
we observe a similar phase error in the physical variable $u_h$ in both cases. However, in the non-uniform case we observe
a larger amount of energy leakage from the physical variable $u_h$ to the auxiliary variable $\phi_h$ 
for both methods 
(A) and (A*), which is larger for method (A*).

We mention that the results presented in Fig.~\ref{fig:ex1-diff}-\ref{fig:ex1-non2} 
are not  peculiar to the lowest order case and 
numerical evidence (not reported in this article) indicate 
a similar behavior on both uniform and non-uniform meshes for $N=1$ and $N=2$.

\begin{figure}[ht!]
 \caption{Numerical solution at time $T=40$ for methods (A).
 Left: $u_h$; Right: $\phi_h$.
Solid line: numerical solution. Dashed line: exact solution.
First row: uniform mesh; Second row: non-uniform mesh.
$N=0$, 20 cells.
}
 \label{fig:ex1-non1}
 \vspace{1ex}
 \includegraphics[width=.45\textwidth,height=.13\textheight]{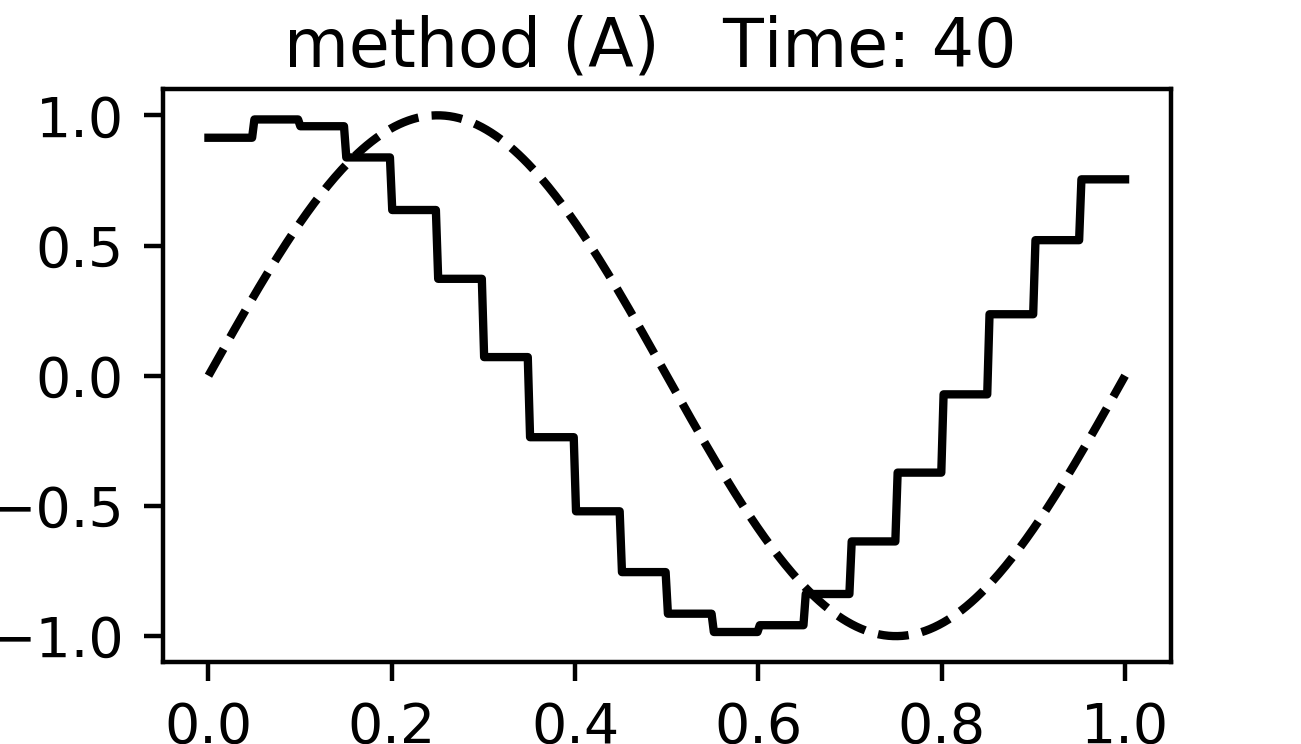}
 \includegraphics[width=.45\textwidth,height=.13\textheight]{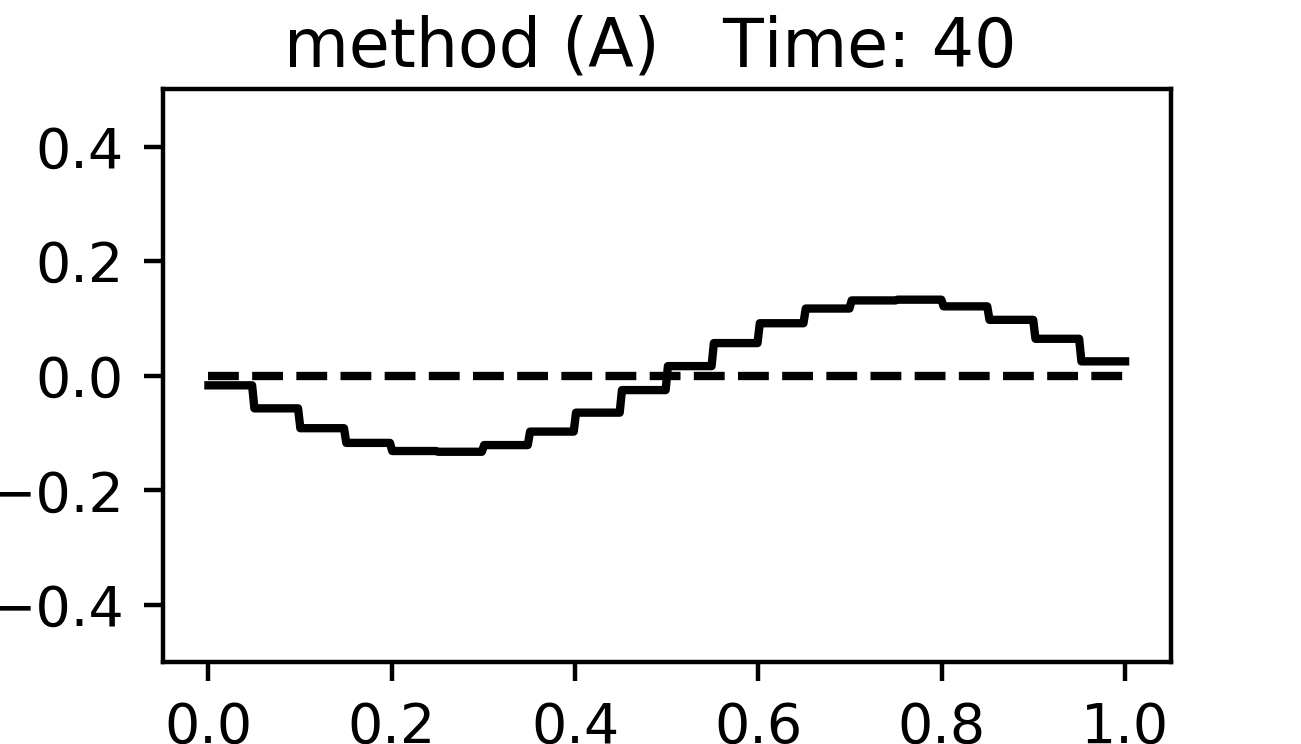}\\[1ex]
 \includegraphics[width=.45\textwidth,height=.13\textheight]{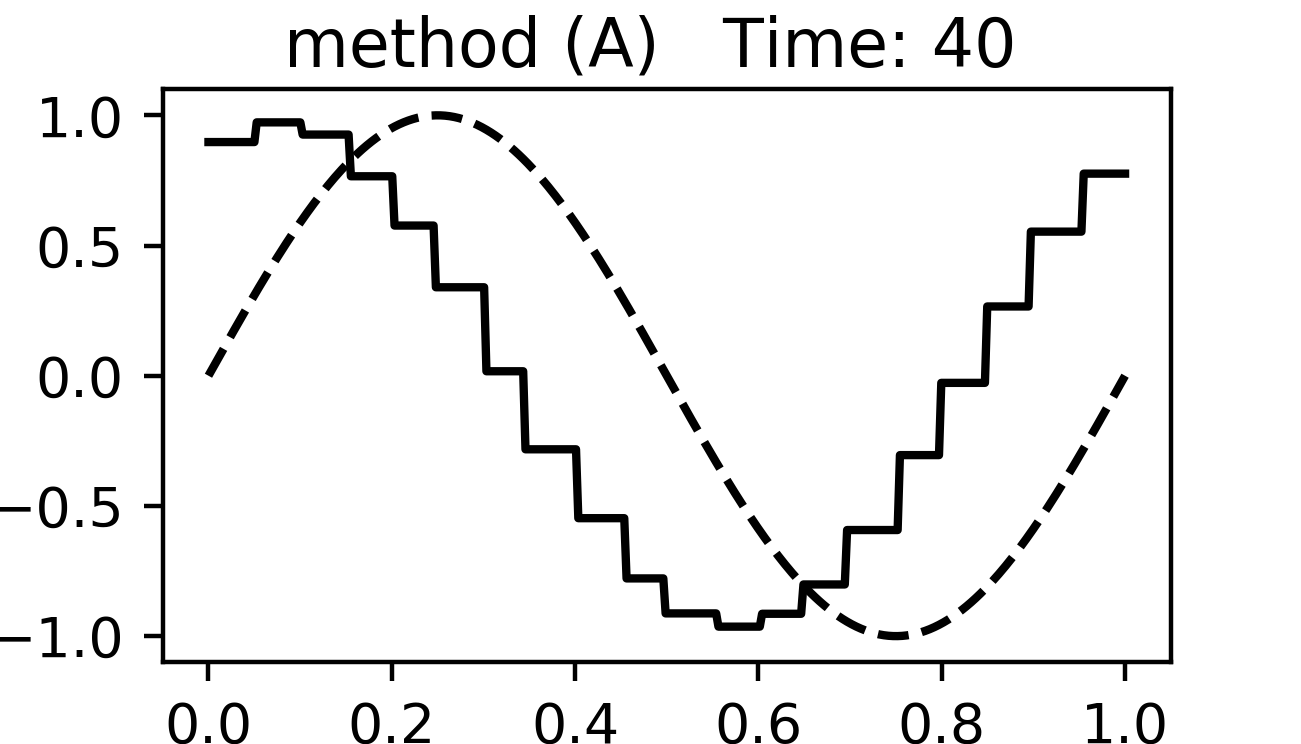}
 \includegraphics[width=.45\textwidth,height=.13\textheight]{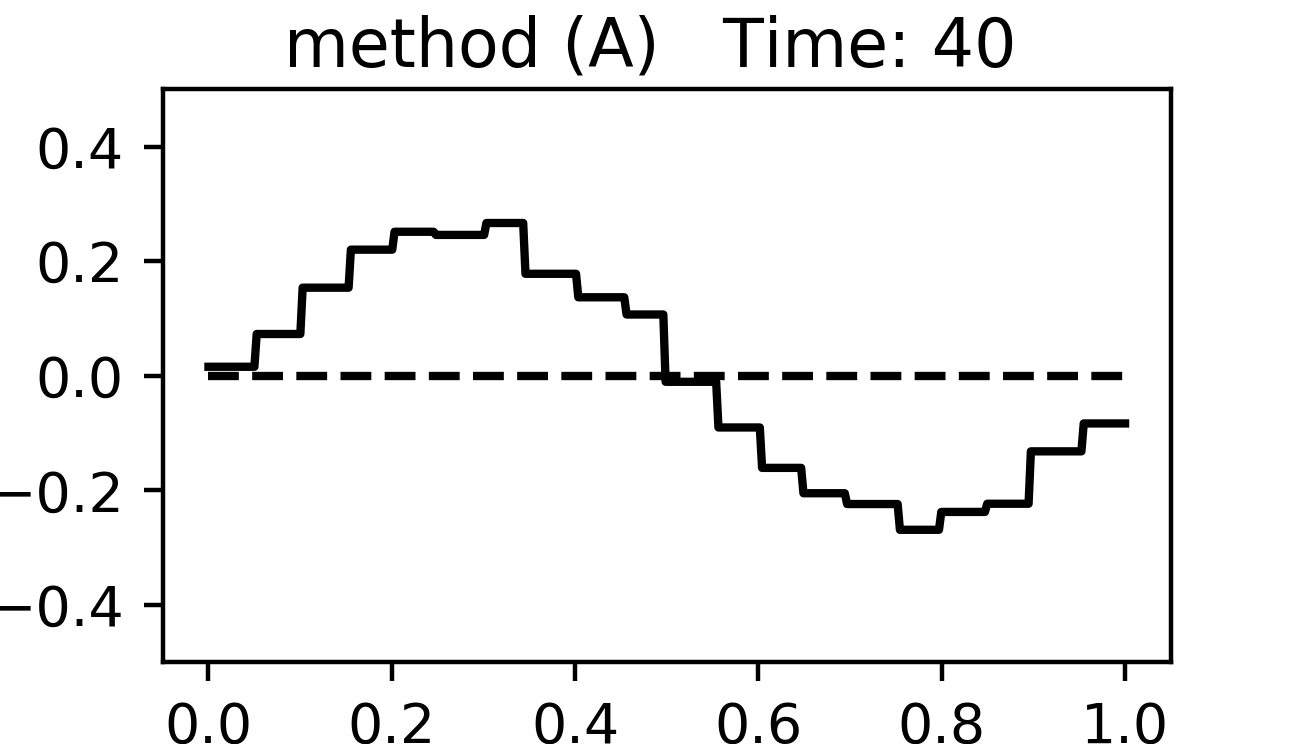}
\end{figure}

\begin{figure}[ht!]
 \caption{Numerical solution at  time $T=40$ for method (A*).
 Left: $u_h$; Right: $\phi_h$.
Solid line: numerical solution. Dashed line: exact solution.
First row: uniform mesh; Second row: non-uniform mesh.
$N=0$, 20 cells.
}
 \label{fig:ex1-non2}
 \vspace{1ex}
  \includegraphics[width=.45\textwidth,height=.13\textheight]{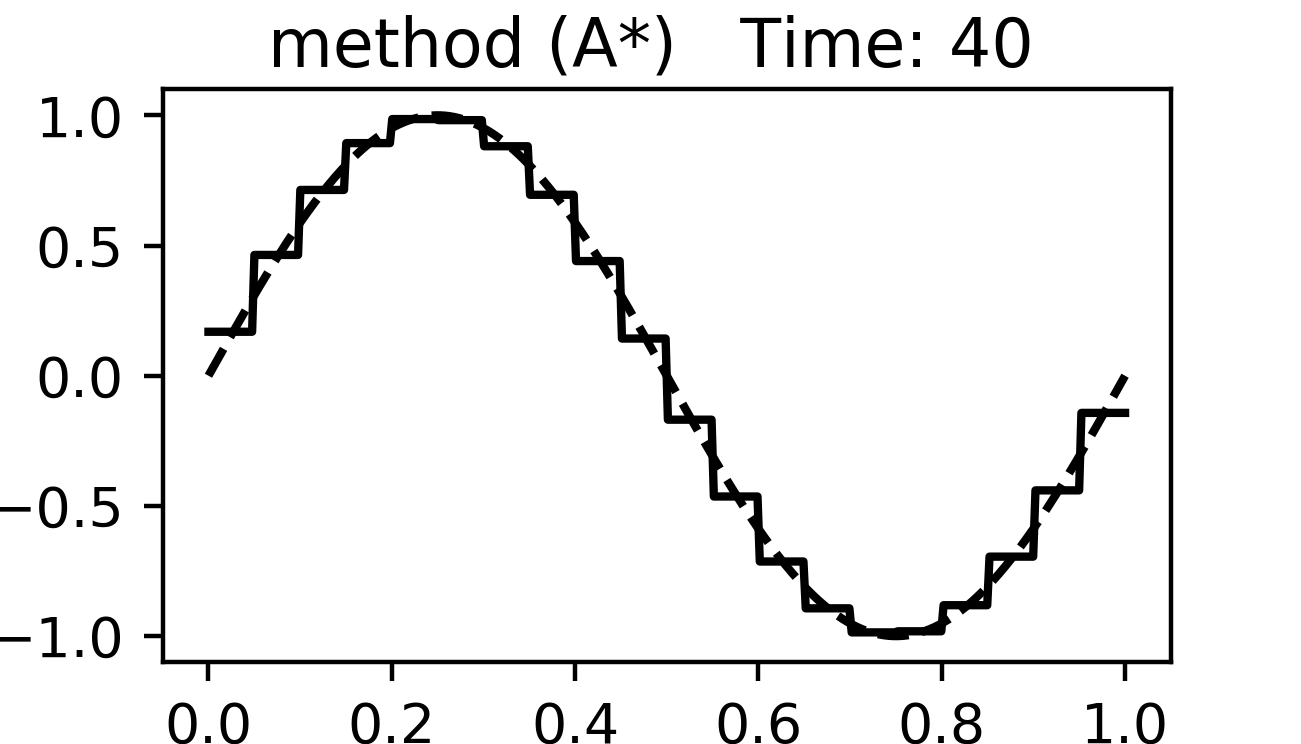}
 \includegraphics[width=.45\textwidth,height=.13\textheight]{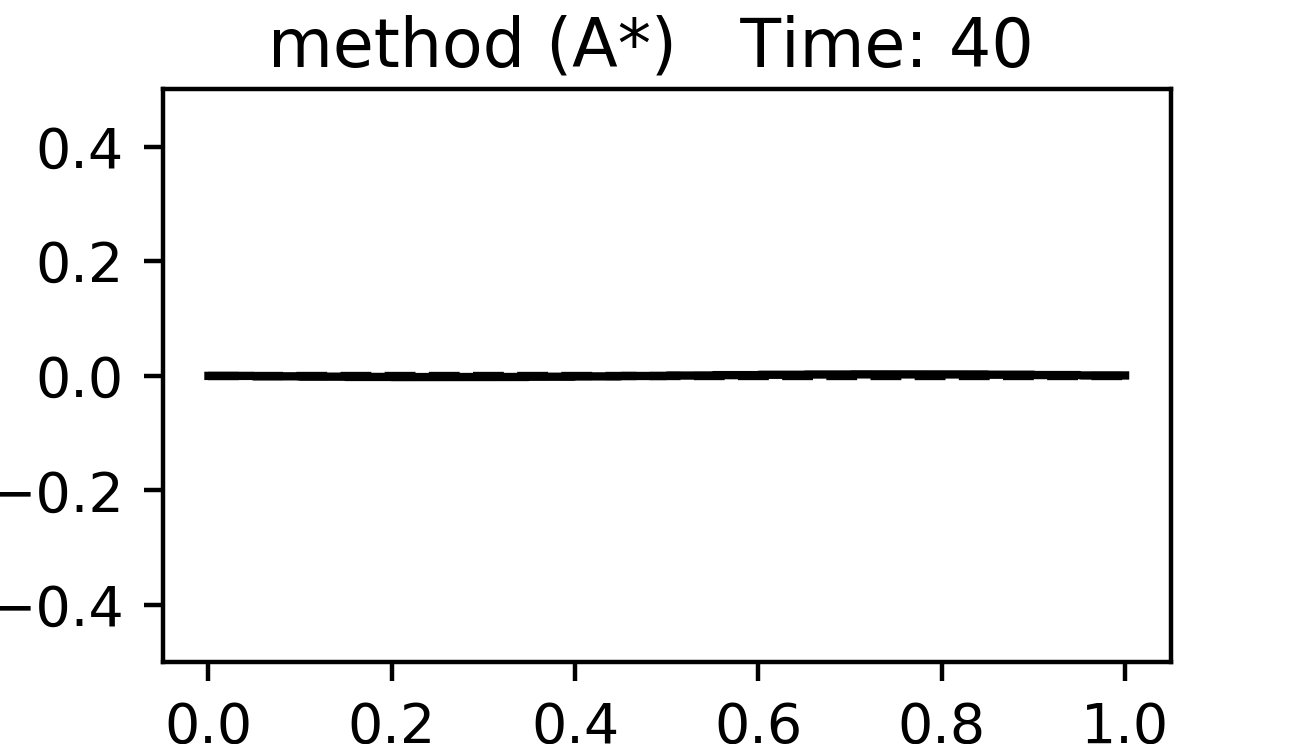}\\[1ex]
  \includegraphics[width=.45\textwidth,height=.13\textheight]{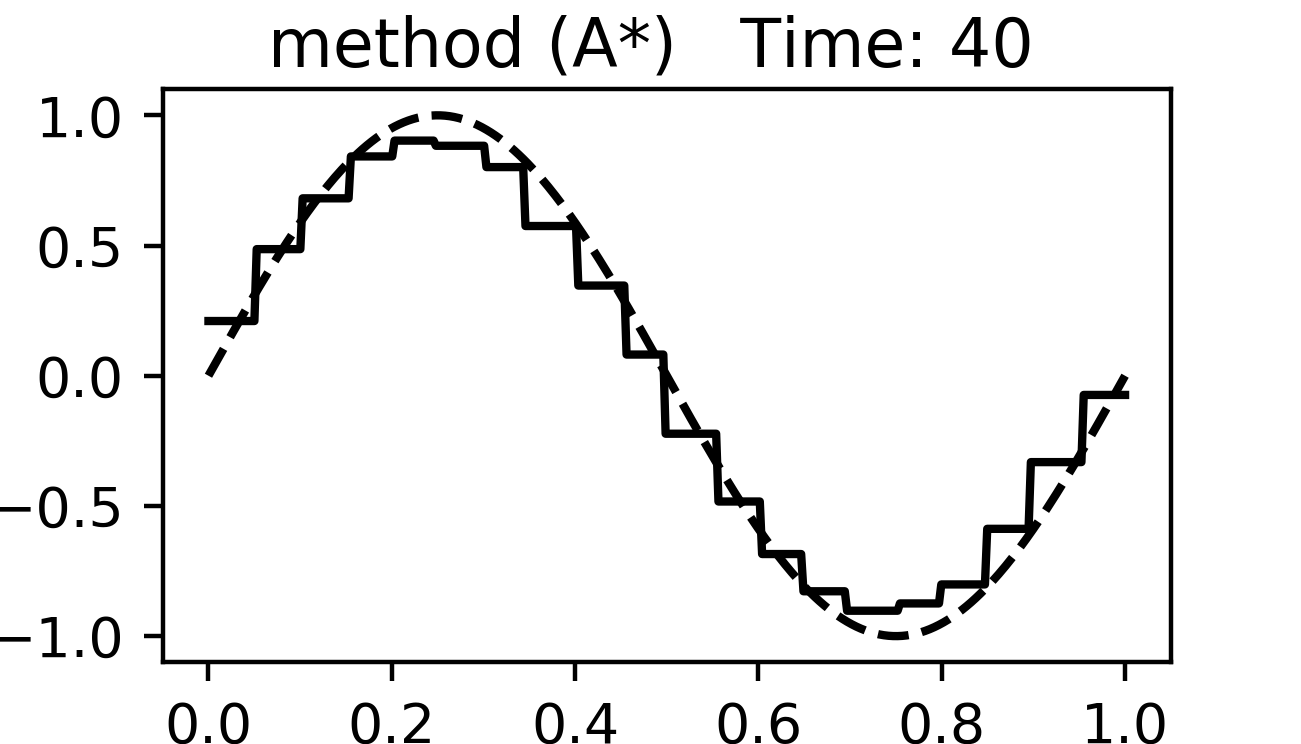}
 \includegraphics[width=.45\textwidth,height=.13\textheight]{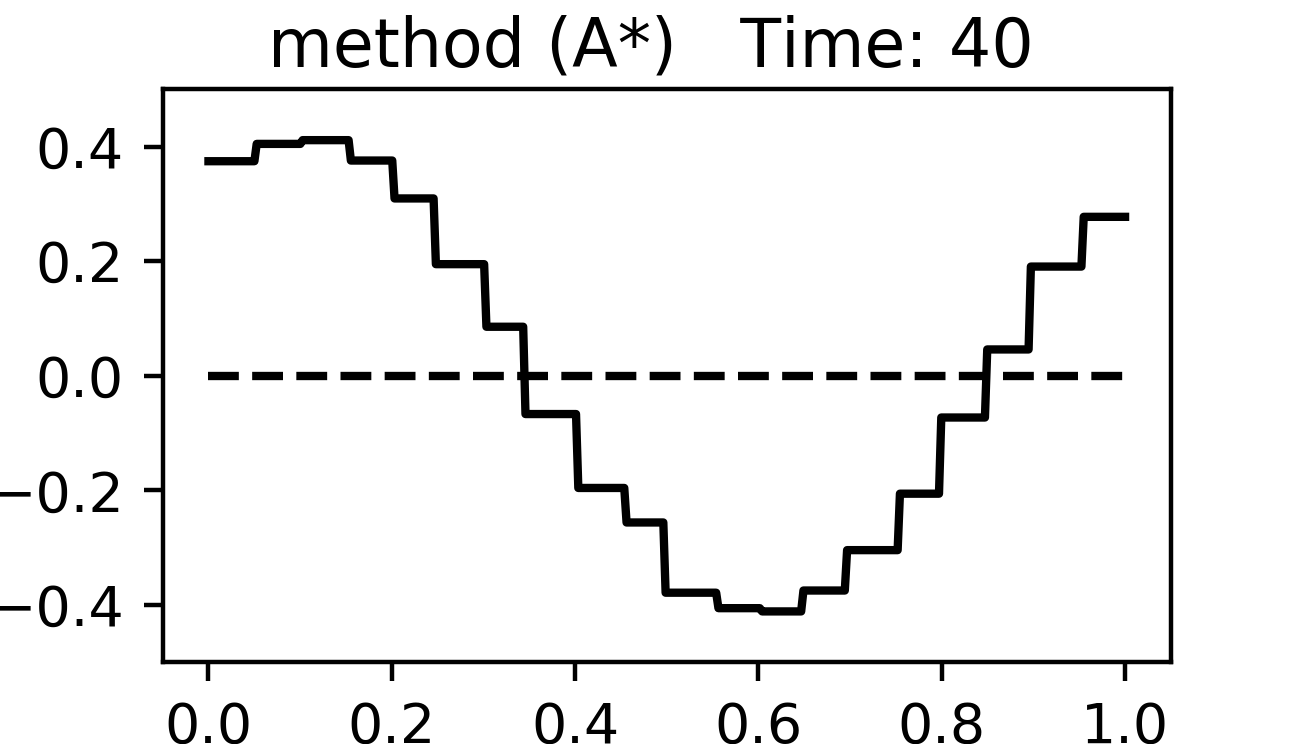}
\end{figure}

\section{Main results on the dispersion analysis}
\label{sec:main}
In this section we provide a theoretical explanation for the improved 
dispersive behavior of methods (A) and (A*) compared with methods (U) and (C).

A key feature of the equations \eqref{aux-adv} 
is the existence of non-trivial, spatially propagating solutions for each given 
temporal frequency $\omega$,
\begin{align}
 \label{uex}
 u(x,t) = \mathrm{e}^{-i\omega t}U(x),\quad\quad
  \phi(x,t) = \mathrm{e}^{-i\omega t}\Phi(x),
\end{align}
where $U(x) = \mathrm{e}^{ikx}$ and $\Phi(x) = \mathrm{e}^{-ikx}$ with $k = \omega$ the wavenumber.
The functions
$U$ and $\Phi$ satisfies a
{\it Bloch-wave condition} 
\begin{align}
 \label{bloch}
 U(x+h) = \lambda^+ U(x), \quad 
 \Phi(x+h) = \lambda^- \Phi(x),\quad x\in\mathbb{R},\, h\in \mathbb{R},
\end{align}
where $\lambda^\pm = \mathrm{e}^{\pm ikh}$ are the Floquet multipliers.

\subsection{The main results}
In order to study the dispersive behavior of the discrete schemes, we 
seek the non-trivial {\it discrete} Bloch wave solutions of the DG scheme \eqref{scheme:adv1d} in the form
\begin{align}
\label{uhex}
u_{h,N}(x,t) = \mathrm{e}^{-i\omega t}U_{h,N}(x),\quad\quad
\phi_{h,N}(x,t) = \mathrm{e}^{-i\omega t}\Phi_{h,N}(x),
\end{align}
where
$U_{h,N}, \Phi_{h,N} \in V_{h}^N$ satisfy a discrete Bloch wave condition
\begin{align}
\label{bloch-h}
 U_{h,N}(x+h) = \lambda_{h,N} U_{h,N}(x), \; 
 \Phi_{h,N}(x+h) = \lambda_{h,N} \Phi_{h,N}(x),\; x\in\mathbb{R},\, h\in \mathbb{R},
\end{align}
and where $\lambda_{h,N}$ is the discrete Floquet multiplier.

The relative accuracy
$R_{h,N}$ of the Floquet multiplier
approximation is defined by
\begin{align}
 \label{err-rel}
 R_{h,N} = \frac{\lambda^+-\lambda_{h,N}}{\lambda^+} = 
 \frac{\mathrm{e}^{ik h}-\lambda_{h,N}}{\mathrm{e}^{ik h}}.
\end{align}

The leading order terms in $R_{h,N}$ for each of the four DG methods described in Section \ref{sec:intro}
are listed in Table \ref{table1}. The results quoted for the methods (U) and (C) are 
special cases of the general 
result proved in  \cite[Theorem 2]{Ainsworth04}, whilst the results for the method (A) are
special cases of the general result that will be proved here in Theorem \ref{thm:disp} and Remark \ref{rk:asym}.
The results for the method (A*) were obtained using algebraic manipulation for particular choices 
of polynomial degree $N$ from $0$ up to degree $17$.

The results given in 
Table \ref{table1} 
show that the accuracy of method (A) is of $(2N+3)$-th order in $\omega h$ and, as such, 
is always superior to the accuracy of methods (U) and (C) both in terms of the order of convergence and the magnitude of the coefficient of the
leading term in the error. 
The method (A*) is better still, providing  
$(2N+5)$-th order of convergence in $\omega h$.


\begin{table}[htbp]
\caption{ \label{table1}
Leading terms of the relative error $R_{h,N}$ in the approximation of the Floquet multiplier for the DG methods
(U), (C), (A), and (A*).
$\Omega=\omega h$.
$\mathsf{C_N} = \frac12\left[\frac{N!}{(2N+1)!}\right]^2$ for $N\ge 1$.
$E_N$, up to four digits accuracy, is given in Table \ref{table-e} for $N\le 17$.
} \centering
\resizebox{\textwidth}{!}{
\begin{tabular}{ccccc}
\noalign{\smallskip}
\hline 
\noalign{\smallskip}
\large Degree &\large{method (U)}&\large{method (C)}&\large{method (A)}&\large{method (A$\!^*$)}\\
\noalign{\smallskip}
 \hline
\noalign{\smallskip}
\LARGE 0 & {\LARGE$\frac{\Omega^2}{2}+i\frac{\Omega^3}{3}$}& \LARGE
$-i\frac{\Omega^3}{6}$
&\LARGE
$-i\frac{\Omega^3}{24}$
&\LARGE
$-i\frac{\Omega^5}{180}$
\\[2ex]\LARGE
1 & \LARGE$\frac{\Omega^4}{72}+i\frac{\Omega^5}{270}$& 
\LARGE$i\frac{\Omega^3}{48}$
&\LARGE
$-i\frac{\Omega^5}{1,080}$
&\LARGE
$-i\frac{53\,\Omega^7}{302,400}$
\\[2ex]
\LARGE2 &\LARGE $\frac{\Omega^6}{7200}+i\frac{\Omega^7}{42,000}$& 
\LARGE$-i\frac{\Omega^7}{16,800}$
&
\LARGE$-i\frac{\Omega^7}{252,000}$
&
\LARGE$-i\frac{41\,\Omega^9}{63,504,000}$
\\[2ex]
{$N\ge 1$} & 
{
$\mathsf{C_N}\!\!\left[\!1+i\frac{(2N+2)\Omega}{(2N+1)(2N+3)\!}\right]
\Omega^{2N+2}$}
& 
{
$i\mathsf{C_N}\left\{\!\!\begin{tabular}{ll}
             $-\frac{N+1}{2N+3}\Omega^{2N+3},$&\!\!\!\!\!$N$ even\\[.6ex]
             $\frac{2N+1}{N+1}\Omega^{2N+1},$&\!\!\!\!\!$N$ odd
            \end{tabular}\right.
$}
&
{
$-i\frac{\mathsf{C_N}\,\Omega^{2N+3}}{(2N+1)(2N+3)}
$}
&
{
$-i\frac{E_N\,{\Omega}^{2N+5}}{(2N+1)^{2N+2}}$}
\\
\noalign{\smallskip}
 \hline
\end{tabular}
}
\end{table}

Let $\mathrm{Re}(\cdot)$ and $\mathrm{Im}(\cdot)$ be the real and imaginary 
parts of a complex number, respectively.
We now examine the dissipation and dispersion errors of the schemes in the low-wavenumber limit where  $kh \ll 1$.
Let $k_{h,N}$ be the {\it discrete wavenumber} that satisfy
\begin{align}
\label{dwave}
 \mathrm{e}^{ik_{h,N}h} = \lambda_{h,N},\quad 
\mathrm{Re}(k_{h,N}h)\in [-\pi,\pi],
\end{align}
which approximates the true wavenumber $k$.
For $kh\ll 1$, the relative error satisfies
\begin{align}
\label{r-e}
R_{h,N}= \frac{\mathrm{e}^{ik h}- \mathrm{e}^{ik_{h,N}h}}{\mathrm{e}^{ik h}}
\approx i(k-k_{h,N})h.
\end{align}
Hence, Table \ref{table1} shows that the {\it dispersion error} is 
\[
 \mathrm{Re}\left((k-k_{h,N})h\right)\approx \left\{
 \begin{tabular}{ll}
$ \frac{\mathsf{C_N}(2N+2)}{(2N+1)(2N+3)}(hk)^{2N+3}$,&  for method (U),\\[2ex]
$ -\frac{\mathsf{C_N}(N+1)}{2N+3}(hk)^{2N+3}$,&  for method (C), even $N$,\\[2ex]
$ \frac{\mathsf{C_N}(2N+1)}{N+1}(hk)^{2N+1}$,&  for method (C), odd $N$,\\[2ex]
$ -\frac{\mathsf{C_N}}{(2N+1)(2N+3)}(hk)^{2N+3}$,&  for method (A),\\[2ex]
$ -\frac{E_N}{(2N+1)^{2N+2}}(hk)^{2N+5}$,&  for method (A*),
 \end{tabular}
 \right.
\]
and the {\it dissipation error} for method (U) is 
\[
 \mathrm{Im}\left((k-k_{h,N})h\right)\approx
\mathsf{C_N}(hk)^{2N+2},
\]
whilst the {dissipation} error for methods (C), (A), and (A*) vanishes
since the discrete wave number $k_{h,N}$ is a real number (due to the fact that 
$|\lambda_{k,N}|=1$ for these methods; see Remark \ref{rk:disp} and Remark \ref{rk:disp2}).


\subsection{Explanation of results presented in Fig.~\ref{fig:ex1-diff}}
Now, let us apply 
the above results (for $N=0$) in Table \ref{table1} to explain the numerical results obtained in Fig.~\ref{fig:ex1-diff}.
For $\Omega=\omega h\ll 1$, 
the numerical solution obtained from each of the DG methods will 
satisfy
\[
 u_h(x,t) \approx \sin(\omega_h x-\omega t)
\]
at the nodes, and the relative error $R_{h,N}\approx i(\omega-\omega_h)h$.
Table \ref{table1} then implies
\begin{alignat*}{2}
 \omega_{h} \approx &\;  \omega +i \frac{\Omega^2}{2h} =
 2\pi + i\frac{\pi^2}{10} &&\quad \text{ for method (U)},\\
 \omega_{h} \approx &\;  \omega +\frac{\Omega^3}{6h} =
 2\pi + \frac{\pi^3}{300} &&\quad \text{ for method (C)},\\
 \omega_{h} \approx &\;  \omega +\frac{\Omega^3}{24h} =
 2\pi + \frac{\pi^3}{1,200} &&\quad \text{ for method (A)},\\
 \omega_{h} \approx &\;  \omega + \frac{\Omega^5}{180h} =
 2\pi + \frac{\pi^5}{900,000} &&\quad \text{ for method (A*)},
\end{alignat*}
where $h=1/20$, and $\Omega = \omega h = \pi/10$.
In particular, the maximum value of the solution at time $T=1$ for method (U) will,
thanks to numerical dissipation, not be unity but will instead take a value close to 
\[
\mathrm{e}^{-\frac{\pi^2}{10}\times 1}
\approx 0.37,
\]
which is in close agreement with the top left figure of Fig.~\ref{fig:ex1-diff}.
The phase lag for method (C) at time $T=5$ will be close to 
$
\frac{\Omega^3}{6 h}\times \frac{5}{\omega}
\approx 0.08
$, 
while for method (A) at time $T=20$ will be close to 
$
\frac{\Omega^3}{24 h}\times \frac{20}{\omega}
\approx 0.08
$, 
and for (A*) at time  $T=1500$ will be close to 
$
\frac{\Omega^5}{180 h}\times \frac{1500}{\omega}\approx 0.08.
$
All of these predictions are 
in close agreement with results in Fig.~\ref{fig:ex1-diff}.
\newcommand{\Pn}{\mathbb{P}_N}

\section{Dispersion analysis: the eigenvalue problem}\label{sec:disp}
In this section, we provide proofs of the dispersion analysis of the semi-discrete scheme \eqref{scheme:adv1d}
leading to the results stated in Table \ref{table1}.
We closely follow the analysis in \cite{Ainsworth04} and begin by seeking 
a  non-trivial bloch-wave solution of the form 
\begin{subequations}
\label{block}
 \begin{align}
  u_{h}(x,t) = \mathrm{e}^{-i\omega t}
  \sum_{m\in\mathbb{Z}}\lambda^m U(x-mh),\\
    \phi_{h}(x,t) = \mathrm{e}^{-i\omega t}
  \sum_{m\in\mathbb{Z}}\lambda^m\Phi(x-mh),
 \end{align}
 \end{subequations}
where $U,\Phi\in V_h^N$.
Denoting $\Omega = \omega h$, 
and transforming the domain over which 
the scheme \eqref{scheme:adv1d} is posed
to the reference interval $[-1,1]$,  we obtain the following eigenvalue problem which 
determines the value of the discrete Floquet multiplier $\lambda$:
%
Find $U, \Phi\in \Pn$ and $\lambda \in \mathbb{C}$ 
such that 
\begin{subequations}
\label{eq:eig}
\begin{alignat}{2}
\label{eq:eig1}
-\frac12 i\Omega(U, v)+(U', v)&+\frac12\Big(
\lambda U(-1) - U(1) + \alpha\big(\lambda\Phi(-1)-\Phi(1)\big)\Big)v(1)&&\;\;\\ 
+
\frac12&\Big(
U(-1) - \lambda^{-1}U(1) -\alpha\big (\Phi(-1)-\lambda^{-1}\Phi(1)\big)\Big)v(-1) && =\; 0\nonumber\\
\label{eq:eig2}
\frac12 i\Omega(\Phi, \psi)+(\Phi', \psi)&+\frac12\Big(
\lambda \Phi(-1) - \Phi(1) + \alpha\big(\lambda U(-1)-U(1)\big)\Big)\psi(1)&&\\ 
+
\frac12&\Big(
\Phi(-1) - \lambda^{-1}\Phi(1) -\alpha\big (U(-1)-\lambda^{-1}U(1)\big)\Big)\psi(-1) && =\; 0,\nonumber
\end{alignat}
\end{subequations}
for all $v,\psi\in\Pn$.
Here $(\cdot,\cdot)$ indicates the $L^2$-inner product on the reference interval $[-1,1]$.
As usual, the condition under which the eigenvalue problem will possess non-trivial solutions reduces to an
algebraic equation for $\lambda$, which we now proceed to identify.

\subsection{Notation and preliminaries}
We denote the differential operators 
\begin{align}
 \label{op}
 \mathcal{L}^\pm(v) := \mp \frac12i\Omega v+v' ,
\end{align}
and recall from \cite{Ainsworth04} the following polynomial functions of degree $N$:
\begin{subequations}
\label{eig-fun}
\begin{align}
\Psi_{N}^{1,\pm}(s) = &\sum_{m=0}^N(\pm i\Omega)^m\frac{(2N+1-m)!}{(2N+1)!}P_m^{(N-m, N-m+1)}(s),\\
\Psi_{N}^{2,\pm}(s) = &\sum_{m=0}^N(\pm i\Omega)^m\frac{(2N+1-m)!}{(2N+1)!}P_m^{(N-m+1, N-m)}(s),
\end{align}
where $P_m^{(p, q)}(s)$ denotes the Jacobi polynomial of type $(p,q)$ and degree $N$.
\end{subequations}
Elementary calculation \cite{Ainsworth04} yields that 
\begin{subequations}
\label{eq-op}
\begin{align}
\label{eq-op1}
 \mathcal{L}^+\Psi_N^{1,+} = &\; -\frac{(i\Omega)^{N+1}}{2}\frac{(N+1)!}{(2N+1)!}P_N^{(0,1)}(s),\\
\label{eq-op2}
 \mathcal{L}^+\Psi_N^{2,+} = &\; -\frac{(i\Omega)^{N+1}}{2}\frac{(N+1)!}{(2N+1)!}P_N^{(1,0)}(s),\\
\label{eq-op3}
 \mathcal{L}^-\Psi_N^{1,-} = &\; -\frac{(-i\Omega)^{N+1}}{2}\frac{(N+1)!}{(2N+1)!}P_N^{(0,1)}(s),\\
\label{eq-op4}
 \mathcal{L}^-\Psi_N^{2,-} = &\; -\frac{(-i\Omega)^{N+1}}{2}\frac{(N+1)!}{(2N+1)!}P_N^{(1,0)}(s),
\end{align}
and standard properties of the Jacobi polynomials reveal that
\begin{align}
\label{eq-op5}
 (\mathcal{L}^+\Psi_N^{1,+},1) = &\;-(\mathcal{L}^-\Psi_N^{2,-},1) =\frac{N!}{(2N+1)!} {(-i\Omega)^{N+1}},\\
\label{eq-op6}
 (\mathcal{L}^+\Psi_N^{2,+},1) = &\;-(\mathcal{L}^-\Psi_N^{1,-},1) = -\frac{N!}{(2N+1)!}{(i\Omega)^{N+1}}.
\end{align}
\end{subequations}

Let ${}_1F_1$ be the confluent hypergeometric function defined by the series
\begin{align}
 \label{hyper}
  {}_1F_1(a, b, z) =\sum_{m=0}^\infty\frac{(a)_m}{(b)_m}\frac{z^m}{m!},
\end{align}
where $(a)_0=1$, and $(a)_m = a (a+1)\cdots (a+m-1)$ denotes the Pochhammer's notation.
To further simplify notation, we denote 
\begin{subequations}
\label{hyper-cst}
\begin{align}
F_N^\pm = &\;
  {}_1F_1(-N, -2N-1, \pm i\Omega),\\
F_{N+1}^\pm = &\;
  {}_1F_1(-N-1, -2N-1, \pm i\Omega),\\
  \label{hyper-cst3}
  \Xi_N = &\;
  \frac{(F_N^-)^2+(F_N^+)^2+(F_{N+1}^-)^2+(F_{N+1}^+)^2}  {F_N^-F_{N+1}^++F_N^+F_{N+1}^-},\\
    \label{hyper-cst4}
  Z_N = &\;
  \frac{(F_N^-)^2-(F_N^+)^2+(F_{N+1}^-)^2-(F_{N+1}^+)^2}  {F_N^-F_{N+1}^--F_N^+F_{N+1}^+}.
\end{align}
\end{subequations}
It is elementary to show that 
\begin{subequations}
\label{cst2}
\begin{alignat}{2}
\label{cst2-1}
\Psi_N^{1,+}(-1) =& \Psi_N^{2,-}(1)&&=
F_{N+1}^- - (-i\Omega)^{N+1}\frac{N!}{(2N+1)!},\\
\label{cst2-2}
\Psi_N^{1,+}(1) =&\Psi_N^{2,-}(-1)
&&=\; F_{N}^+,\\
\label{cst2-3}
\Psi_N^{2,+}(-1) =&
\Psi_N^{1,-}(1) 
&&=\; F_N^- ,\\
\label{cst2-4}
\Psi_N^{2,+}(1) =&
\Psi_N^{1,-}(-1)
&&=\; F_{N+1}^+ - (i\Omega)^{N+1}\frac{N!}{(2N+1)!}.
\end{alignat}
\end{subequations}
It is also easy to verify that $\Xi_N$ is a real number 
and $Z_N$ is a purely imaginary number for $\Omega\in \mathbb{R}$.
Finally, we denote the constants
\begin{subequations}
 \label{roots}
 \begin{align}
 \label{roots1}
 \lambda_N^\pm = \frac12(\Xi_N\pm\sqrt{\Xi_N^2-4}),
 \end{align}
 and
 \begin{align}
 \label{roots2}
 \mu_N^\pm = \frac12(Z_N\mp\sqrt{Z_N^2+4}),
 \end{align}
\end{subequations}
corresponding to 
the pairs of roots of the quadratic equations
$
 \lambda^2-\Xi_N \lambda +1 = 0,
$
and 
$
 \mu^2-Z_N \mu -1 = 0,
$
respectively.

\subsection{Conditions for an eigenvalue. Case $\bld{\alpha=1}$}
We first consider the case $\alpha = 1$ in the numerical fluxes \eqref{flux:2x2}, 
which corresponds to method (A).
Our main result for the eigenvalue problem \eqref{eq:eig} in this case 
is summarized as follows:
\begin{theorem}
 \label{thm:eig}
 There exists a non-trivial Bloch wave solution of the form \eqref{block} 
 for the scheme \eqref{scheme:adv1d} with numerical fluxes \eqref{flux:2x2} 
 with $\alpha=1$ if and only if $\lambda = \lambda_N^\pm$ with 
$\lambda_N^\pm$ given in \eqref{roots1}. 
%
%
\end{theorem}
\begin{proof}
The proof is elementary and follows a similar path to \cite[Lemma 3]{Ainsworth04}.
We assume the polynomial degree $N\ge 1$ (the lowest order case $N=0$ can be verified easily as a special case).

We shall prove that $\lambda=\lambda_N^\pm$ are the {\it only} two eigenvalues of the problem \eqref{eq:eig}.
To this end, let $\lambda$ be an eigenvalue of \eqref{eq:eig} with $\alpha = 1$, 
with $(U,\Phi)\in \Pn\times\Pn$ corresponding (non-trivial) eigenfunctions.
Equation \eqref{eq:eig1} implies that 
\[
 (\mathcal{L}^+U, v) = 0,\quad \forall v = (1-s)(1+s)w, \text{ with }w \in \mathbb{P}_{N-2},
\]
and hence, since $\mathcal{L}^+U\in \Pn$, 
we obtain
\[
 \mathcal{L}^+U = \tilde a^+ P_N^{(0,1)}+\tilde b^+ P_N^{(1,0)},
\]
where $\tilde a^+,\tilde b^+\in\mathbb{C}$ are constants to be determined.
Using the fact that $\mathcal{L^+}:\Pn\rightarrow \Pn$ is one-to-one along with \eqref{eq-op},
we get 
\begin{align}
\label{form-u}
U =  a^+ \Psi_N^{1,+}+ b^+ \Psi_N^{2,+}.
\end{align}
Similar, we have 
\begin{align}
\label{form-p}
\Phi =  a^- \Psi_N^{1,-}+ b^- \Psi_N^{2,-}, 
\end{align}
with  $ a^-, b^-\in\mathbb{C}$ constants to be determined.
Now, taking $v=1-s$ and $\psi=1-s$ in equations \eqref{eq:eig} and adding, we get
\[
0= (\mathcal{L}^+U, 1-s)+
  (\mathcal{L}^-\Phi, 1-s)
  =
  2 (-i\Omega)^{N+1}\frac{N!}{(2N+1)!}( a^+-(-1)^N a^-), 
\]
which implies that 
\[
  a^-=(-1)^N a^+.
\]
Similarly, 
take $v=1+s$ and $\psi=-(1+s)$ in equations \eqref{eq:eig} and adding, we get
\[
0= (\mathcal{L}^+U, 1+s)-
  (\mathcal{L}^-\Phi, 1+s)
  =
  -2 (i\Omega)^{N+1}\frac{N!}{(2N+1)!}( b^++(-1)^N b^-), 
\]
which implies that 
\[
  b^-=-(-1)^{N} b^+.
\]
Hence, 
\[
 \Phi = (-1)^N( a^+ \Psi_N^{1,-}- b^+ \Psi_N^{2,-}).
\]
Without loss of generality, we assume that $ a^+ = 1$, and denote $\mu =  b^+/ a^+ =  b^+$.
Thus, we have identified the eigenfunctions. In order to identify the eigenvalues, we choose
test function $v = 1-s$, and $v=1+s$ in equation \eqref{eq:eig1}, respectively. 

Using \eqref{cst2}, elementary calculation yields that
\begin{subequations}
\begin{align}
U(-1)+ 
\Phi(-1) =  &\;(F_{N+1}^-+(-1)^NF_{N+1}^+)
+ \mu (F_{N}^--(-1)^NF_{N}^+),\\
U(-1)-
\Phi(-1) =  &\;(F_{N+1}^--(-1)^NF_{N+1}^+)
+ \mu (F_{N}^-+(-1)^NF_{N}^+)\nonumber\\
&\;\;-2(-i\Omega)^{N+1}\frac{N!}{(2N+1)!},\\
U(1)+
\Phi(1) =  &\;(F_{N}^++(-1)^NF_{N}^-)
+ \mu (F_{N+1}^+-(-1)^NF_{N+1}^-)\nonumber\\
&\;\;-2\mu(i\Omega)^{N+1}\frac{N!}{(2N+1)!},\\
U(1)-
\Phi(1) =  &\;
(F_{N}^+-(-1)^NF_{N}^-)
+ \mu (F_{N+1}^++(-1)^NF_{N+1}^-).
\end{align}
\end{subequations}
Combing the above identities with \eqref{eq-op5} and \eqref{eq-op6}, 
equation \eqref{eq:eig1} with $v=1-s$ 
reduces to an algebraic equation for $\lambda$ and $\mu$
\begin{align*}
&(F_{N+1}^--(-1)^NF_{N+1}^+)
+ \mu (F_{N}^-+(-1)^NF_{N}^+)\\
&-\lambda^{-1}\Big(
(F_{N}^+-(-1)^NF_{N}^-)
+ \mu^+ (F_{N+1}^++(-1)^NF_{N+1}^-)
\Big)=0,
\end{align*}
whilst equation \eqref{eq:eig1} with 
$v=1+s$ gives a second algebraic equation
\begin{align*}
&\lambda(F_{N+1}^-+(-1)^NF_{N+1}^+)
+ \mu (F_{N}^--(-1)^NF_{N}^+)\\
&-\Big(
(F_{N}^++(-1)^NF_{N}^-)
+ \mu (F_{N+1}^+-(-1)^NF_{N+1}^-)
\Big)=0.
\end{align*}
Simplifying leads to the algebraic system
\begin{align*}
\lambda (F_{N+1}^-+ \mu F_{N}^-)-(F_{N}^++ \mu F_{N+1}^+)=0,\\
\lambda (F_{N+1}^+- \mu F_{N}^+)-(F_{N}^-- \mu F_{N+1}^-)=0.
\end{align*}
Eliminating $\mu$ then gives
\[
 \lambda^2 - \Xi_N \lambda +1 = 0,
\]
while eliminating $\lambda$ gives
\[
 \mu^2 - Z_N \mu -1 = 0,
\]
with $\Xi_N$ and $Z_N$ given in \eqref{hyper-cst3} and \eqref{hyper-cst4}, respectively.
Hence, $\lambda=\lambda_N^\pm$, with the corresponding $\mu = \mu_N^\pm$ given in \eqref{roots}.
This completes the proof.
%
\end{proof}
\begin{remark}[Matrix-vector form]
Given a set of basis functions of $\Pn$,
one can directly formulate eigenvalue problem \eqref{eq:eig} in matrix-vector form
\[
 A(\lambda)\mathbf{x} = 0, 
\]
where
$A(\lambda)\in \mathbb{R}^{2(N+1)\times 2(N+1)}$ and $\mathbf{x}\in\mathbb{R}^{2(N+1)}$. 
A non-trivial solution exists if and only if the determinant of $A(\lambda)$ vanishes.
Theorem \ref{thm:eig} shows that, after proper normalization,
\begin{align}
\label{det}
 \det(A(\lambda)) = \lambda^2 -\Xi_N\lambda +1. 
\end{align}
\end{remark}

\newcommand{\En}{\mathcal{E}_N}
\subsection{Properties of the eigenvalues. Case $\bld{\alpha= 1}$}
\label{sec:eig}
The next result characterises the solutions of the algebraic eigenvalue equation as approximations to the
modes $\{\lambda_N^+,\lambda_N^-\} \approx \{\mathrm{e}^{i\Omega}, \mathrm{e}^{-i\Omega}\}$.
It will be shown that $\lambda_N^+$ approximates the mode $\mathrm{e}^{i\Omega}$ if $\sin(\Omega)\ge 0$, while 
it approximates the mode  $\mathrm{e}^{-i\Omega}$ if $\sin(\Omega)< 0$.
Thus, it is convenient to define
\[
 \lambda_N^\pm = \left\{\begin{tabular}{ll}
$ \frac{\Xi_N\pm \sqrt{\Xi_N^2-4}}{2}$& 
$ \text{ if } \sin(\Omega)\ge0,$\\[2ex]
$ \frac{\Xi_N\mp \sqrt{\Xi_N^2-4}}{2}$& 
$ \text{ if } \sin(\Omega)<0,$
                        \end{tabular}\right.
\]
so that the algebraic eigenvalue $\lambda_N^+$ always approximates the {\it positive} mode $\mathrm{e}^{i\Omega}$.
We denote the relative error
\begin{align}
\label{rho}
 \rho_N^\pm = \frac{\mathrm{e}^{\pm i\Omega}-\lambda_N^\pm}{ \mathrm{e}^{\pm i\Omega}}.
\end{align}
It was shown in \cite{Ainsworth04} in the case of upwinding scheme (U) and the 
centered flux scheme (C) that
the relative error $\rho_{N}^\pm$ is dictated by the remainder in certain Pad\'e approximants 
of the exponential. The following result shows that the accuracy of the {\it same} Pad\'e approximants dictates
the error in the scheme (A):
\begin{theorem}
 \label{thm:disp}
There holds
\begin{align}
 \label{est-x}
 \Xi_N  = 2\cos(\Omega) + 
 2\Theta_N\sin(\Omega)
 +\mathcal{O}(|\En|^2),
\end{align}
where $\Theta_N\in\mathbb{R}$ is given by 
\begin{align}
\label{theta}
\Theta_N= \mathrm{Im}(\En)+ 
   \mathrm{Re}(\En)\frac{\mathrm{Im}((F_N^-)^2\mathrm{e}^{i\Omega})}{\mathrm{Re}((F_N^-)^2\mathrm{e}^{i\Omega})}, 
\end{align}
and
 \begin{align}
 \label{pade}
\En = \frac{\mathrm{e}^{i\Omega}-[N+1/N]_{\mathrm{e}^{i\Omega}}}{\mathrm{e}^{i\Omega}},
\end{align}
with
$
 [N+1/N]_{\mathrm{e}^{i\Omega}} = \frac{F_{N+1}^+}{F_{N}^-}
$
being the $[N+1/N]$-Pad\'e approximant of $\mathrm{e}^{i\Omega}$.

Moreover, there holds
\begin{align}
 \label{est-lam}
 \rho_{N}^\pm  =\pm i\Theta_N + \mathcal{O}(|\En|^2).
\end{align}
\end{theorem}
\begin{proof}
To ease the notation, we denote 
\begin{align}
\label{hn}
 H_N = (F_N^{-})^2\mathrm{e}^{i\Omega}. 
\end{align}
We first obtain the estimate \eqref{est-x}.
 By the definition of $\En$ in \eqref{pade}, we have 
 \[
  F_{N+1}^+ = F_N^{-}\mathrm{e}^{i\Omega}(1-\En),
 \]
 and by definition of the constants in \eqref{hyper-cst}, we have 
\[
  F_{N}^+ = \mathrm{Conj}(F_N^-),\quad 
 F_{N+1}^- = \mathrm{Conj}(F_{N+1}^+).
 \]
Applying the above expressions to \eqref{hyper-cst3} and simplifying, we get
\begin{align*}
 \Xi_N = &\;2\frac{\mathrm{cos}(\Omega)\,\mathrm{Re}(H_N)-\mathrm{Re}(H_N\mathrm{e}^{i\Omega}\En)
 +\frac12\mathrm{Re}(H_N\mathrm{e}^{i\Omega}\En^2)
 }{\mathrm{Re}(H_N) - \mathrm{Re}(H_N\En)}.
\end{align*}
We then get the estimate \eqref{est-x} by  performing a series expansion in $\En\ll 1$ of the above right hand side.

Using the definition of $\lambda_N^\pm$ in \eqref{roots1}, we  obtain
\begin{align*}
 \lambda_N^\pm = &\;(\cos(\Omega)+\Theta_N \sin(\Omega))
 \pm i(\sin(\Omega)-\Theta_N\cos(\Omega))+\mathcal{O}(|\mathcal{E}_N|^2)\\
 = &\; \mathrm{e}^{\pm i\Omega}(1\mp i\Theta_N)+\mathcal{O}(|\mathcal{E}_N|^2).
\end{align*}
The estimate \eqref{est-lam} now follows directly from 
 definition \eqref{rho}.
\end{proof}

\begin{remark}[Asymptotic behavior of the remainder $\rho_N^+$]
\label{rk:asym}
Series expansion in $\Omega$ for the expression
$\frac{\mathrm{Im}((F_N^-)^2\mathrm{e}^{i\Omega})}{\mathrm{Re}((F_N^-)^2\mathrm{e}^{i\Omega})}$  reveals that 
\begin{align*}
\frac{\mathrm{Im}((F_N^-)^2\mathrm{e}^{i\Omega})}{\mathrm{Re}((F_N^-)^2\mathrm{e}^{i\Omega})}
  =&\; \frac{\Omega}{2N+1}
  \left(1+C_N\left[\frac{\Omega}{2N+1}\right]^2+\mathcal{O}\left(\left[\frac{\Omega}{2N+1}\right]^4\right)\right),
\end{align*}
with 
$C_N=\left\{\begin{tabular}{ll}
         $1/3$ & $N=0$,\\[.4ex]
         $\frac{N}{2N-1}$ & $N>0$.         
        \end{tabular}\right.$ 
Combing this estimate with \eqref{est-lam}, we obtain
\begin{align}
\label{est-asy}
 \rho_N^+ = i\left( \mathrm{Im}(\En)+ 
   \mathrm{Re}(\En)\frac{\Omega}{2N+1}\right)
   +\mathcal{O}\left(\mathrm{Re}(\En)\left[\frac{\Omega}{2N+1}\right]^3+|\En|^2\right).
\end{align}
It remains to estimate $\En$. This was discussed in detail in \cite[Section 3]{Ainsworth04}
in the cases where $\Omega\ll 1$ and where $N\rightarrow \infty$.
In particular, \cite[Corollary 1]{Ainsworth04} gives that, for $\Omega\ll 1$:
\begin{align*}
  \mathrm{Re}(\En)=&\; -\frac{\Omega^{2N+2}}{2}\left[\frac{N!}{(2N+1)!}\right]^2
+\mathcal{O}(\Omega^{2N+4}),\\
\mathrm{Im}(\En)=&\; i\Omega^{2N+3}\frac{N+1}{(2N+1)(2N+3)}\left[\frac{N!}{(2N+1)!}\right]^2
 +\mathcal{O}(\Omega^{2N+5}).
\end{align*}
Hence, 
for  $\Omega\ll 1$, we have
 \begin{align}
 \label{small-k}
\rho_N^+ = -i D_N \Omega^{2N+3}+\mathcal{O}(\Omega^{2N+5}),
 \end{align}
where $D_N=\left\{\begin{tabular}{ll}
         $1/24$ & $N=0$,\\[.4ex]
         $\frac{1}{2(2N+1)(2N+3)}\left[\frac{(N)!}{(2N+1)!}\right]^2$ & $N>0$.         
        \end{tabular}\right.$

The behavior in the case when $\Omega$ is fixed and $N\rightarrow \infty$ is more subtle.
In particular,  $\rho_N^+$ passes through three distinct phases \cite{Ainsworth04}:
\begin{itemize}
 \item [(1)] If $2N+1<\Omega - C \Omega^{1/3}$, $\rho_N^+$ oscillate but do not decay;
 \item [(2)] If $\Omega - o(\Omega^{1/3})<2N+1<\Omega + o(\Omega^{1/3})$, $\rho_N^+$ decays algebraically at a rate
 $\mathcal{O}(N^{-1/3})$;
 \item [(3a)] If $N,\Omega \rightarrow\infty$ in such a way that  
 $2N+1 = \kappa \Omega$ with $\kappa>1$ fixed, then 
 $\rho_N^+$ decays exponentially:
 \[
  \rho_N^+\approx i \mathrm{e}^{-\beta(N+1/2)}\left(1-\sqrt{1-\frac{1}{\kappa^2}}\right)^2,
 \]
 where $\beta>0$ is given by 
 \[
  \beta = \mathrm{ln}\frac{1+\sqrt{1-\frac{1}{\kappa^2}}}{1+\sqrt{1-\frac{1}{\kappa^2}}}
  -2\sqrt{1-\frac{1}{\kappa^2}};
 \]

 \item [(3b)] If $2N+1\gg \Omega$, then $\rho_N^+$ decays at a super-exponential rate:
 \[
\rho_N^+\approx -i \left[\frac{\mathrm{e}\Omega}{2\sqrt{(2N+1)(2N+3)}}\right]^{2N+2}\frac{2\Omega}{(2N+1)(2N+3)}.
 \]
\end{itemize}
\end{remark}
\begin{remark}[Dissipation error for small $\Omega$]
\label{rk:disp}
Series  expansion of $\Xi_N$ in $\Omega \ll 1$ yields that 
\[
 \Xi_N = 2-\Omega^2 +\mathcal{O}(\Omega^4).
\]
Hence, $|\Xi_N|<2$ for $\Omega\ll 1$, which implies that 
the two eigenvalues $\lambda_N^\pm$ are complex-conjugates and have unit  modulus. 
In particular, this means that method (A)
is non-dissipative.
\end{remark}

\newcommand{\am}{F_N^-}
\newcommand{\ap}{F_N^+}
\newcommand{\bm}{F_{N+1}^-}
\newcommand{\bp}{F_{N+1}^+}
\subsection{Conditions for an eigenvalue. General $\bld{\alpha}$}
Now we consider the case with a general
value of the parameter
$\alpha$ in the numerical fluxes \eqref{flux:2x2}. 
Our main result for the eigenvalue problem \eqref{eq:eig} in this case 
is summarised in the following theorem.
\begin{theorem}
 \label{thm:eig2}
 There exists a non-trivial Bloch wave solution of the form \eqref{block} 
 for the scheme \eqref{scheme:adv1d} with numerical fluxes \eqref{flux:2x2}
 if and only if $\lambda$ is a root of the  
algebraic equation 
\begin{align}
\label{quad}
\frac{1}{\lambda^2}\det(M(\lambda)) = {a_N\left(\lambda+\frac{1}{\lambda}\right)^2
+
b_N\left(\lambda+\frac{1}{\lambda}\right)+c_N} = 0,
\end{align}
where 
\begin{align*}
 a_N = &\; (-1)^N(1-\alpha^2)\am\ap,\\
 b_N = &\; -(-1)^N(1-\alpha^2)(\am\bm+\ap\bp)
 +(1+\alpha^2)(\ap\bm+\am\bp),\\
 c_N = &\; 2(-1)^N(1-\alpha^2)(\bm\bp-\am\ap)\\
 &\; -(1+\alpha^2)((\am)^2+(\ap)^2+(\bm)^2+(\bp)^2),
\end{align*}
are real constants,
and $M(\lambda)$ is the matrix 
\begin{align*}
M(\lambda) = 
\left[ \begin{tabular}{cccc}
  $\lambda F_{N+1}^--F_N^+$
  &
   $\lambda F_{N}^--F_{N+1}^+$&$0$&$0$\\
$0$&$0$&  $\lambda F_{N+1}^+-F_N^-$
  &
   $\lambda F_{N}^+-F_{N+1}^-$\\
  $\lambda (-1)^N$
  & $-1$&$-\alpha \lambda$&$-\alpha(-1)^N$\\
  $\alpha\lambda (-1)^N$
  & $\alpha$&$-\lambda$&$(-1)^N$\\
 \end{tabular}\right].
\end{align*}
\end{theorem}
\begin{proof}
The proof is similar to that of Theorem \ref{thm:eig}, and we only sketch the main differences.
Let $\lambda$ be an eigenvalue of \eqref{eq:eig}, with $(U,\Phi) \in \Pn \times \Pn$ 
the corresponding (non-trivial) eigenfunctions.
As before, using the fact that 
\[
(\mathcal{L}^+ U, v) = (\mathcal{L}^- \Phi, v) = 0,\quad 
\forall v = (1-s)(1+s)w,\text{ with }w\in \mathbb{P}_{N-2},
\]
we obtain
\[
 U = a^+ \Phi_N^{1,+}
 +b^+ \Phi_N^{2,+}, \quad 
  \Phi = a^- \Phi_N^{1,-}
 +b^- \Phi_N^{2,-}.
\]
The coefficients $a^\pm$ and $b^\pm$ must now satisfy the {\it four} algebraic equations 
corresponding to choosing test functions in
\eqref{eq:eig} of the form $v=1\pm s$ and $\phi=1\pm s$. This leads to 
a $4\times 4$ system of homogeneous linear equations for the vector $\bld {\mathrm{x}}=[a^+,b^+,a^-,b^-]^T$. 
By straightforward 
but tedious algebraic manipulation, we arrive at the system of equations 
$
 M(\lambda)\bld {\mathrm{x}} = 0,
$
where $M(\lambda)$ is defined above.
\end{proof}

\begin{remark}[Spurious modes]
Note that, in  the case $\alpha=1$, the 
equation \eqref{quad} is a {\it linear} function for the variable 
$z = \lambda + 1/\lambda$, which results in two roots (approximating the two  
physical modes $\mathrm{e}^{\pm i\Omega}$).
However, in the general case with $|\alpha| \not=1$, the equation \eqref{quad} 
is {\it quadratic} in $z$ leading to $4$ roots. Two of these roots will 
approximate the physical modes $\mathrm{e}^{\pm i\Omega}$, 
while the remaining two roots correspond to 
{\it spurious} modes.
The presence of spurious modes in numerical schemes for wave equations is well-known: in \cite{Ainsworth04}
it was shown that the centered DG method (C) also has a spurious mode.
A precise characterisation of these eigenvalues
similar to the case $\alpha=1$ discussed in subsection 
\ref{sec:eig} for any $|\alpha|\not = 1$ 
is rather 
technical to derive and is not pursued further here;
see, for example, 
in \cite{Ainsworth04} the discussion on central DG method ($\alpha=0$).
\end{remark}

\begin{remark}[Dissipation error for small $\Omega$]
\label{rk:disp2}
When $\Omega\ll 1$, we show in the following that, if 
 $\alpha\not\in \{0,\pm1\}$, then 
two of the four roots of the equation \eqref{quad} 
are complex-conjugate to each other and have  modulus $1$, which approximate 
the physical modes $\mathrm{e}^{\pm i\Omega}$, 
and the other two are real, which are non-physical.
Hence, the method is non-dissipative.

Denoting $
 f(z) = a_Nz^2+b_N z + c_N,
$
series expansion on $\Omega\ll 1$ yields that 
\[
 f(2) f(-2) = -64 \alpha^2 \Omega^2 +\mathcal{O}(\Omega^4).
\]
This implies that $ f(2) f(-2)<0$ for $\Omega\in\mathbb{R}^+$ small enough.
Hence, the quadratic equation $f(z)=0$ has two real roots $z_1, z_2$, with 
$|z_1|<2$ and $|z_2|>2$.
This implies that the four roots of the equation \eqref{quad} are determined by the following two quadratic 
equations:
\[
 \lambda+1/\lambda_1=z_1, \text{ or }
  \lambda+1/\lambda_1=z_2.
\]
Since $|z_1|<2$, the two roots of the equation $ \lambda+1/\lambda_1=z_1$ are complex-conjugate to each 
other with modulus $1$.
Since $|z_2|>2$, the two roots of the equation $ \lambda+1/\lambda_1=z_2$ are real.
\end{remark}

\begin{remark}[Leading terms of the relative error $\rho_{N}^+$ for $\alpha$ in \eqref{alpha-opt}] 
\label{rk:asym2}
Remark \ref{rk:asym}  shows that the leading term in the relative error $\rho_N^+$ is of order 
$\Omega^{2N+3}$ for $\alpha=1$. 
Intuitively, one might expect be able to get an even higher order leading term for the relative 
error through a judicious choice of the parameter $\alpha$.
This was shown to be the case in \cite{AinsworthMonk06} for DG methods for two-wave wave equations.

Symbolic manipulation for degree up to $N=17$ demonstrates that, with $\alpha$ given \eqref{alpha-opt},
the
relative error enjoys an additional two orders of accuracy
\[
 \rho_N^+ = -i \frac{E_N}{(2N+1)^{2N+2}}{\Omega}^{2N+5}+\mathcal{O}(\Omega^{2N+7}) 
\]
with the coefficient $E_N$ up to 4 digits accuracy
given in the following table for $N\le 17$:
\begin{table}[htbp]
\caption{
Coefficients $E_N$ up to four digits accuracy for $N\le 17$.
}
\label{table-e}
\centering
\begin{tabular}{ccccccc}
\noalign{\smallskip}
\hline 
\noalign{\smallskip}
Degree&0&1&2&3&4&5\\ [1ex]
$E_N$&
5.555e-03&   1.419e-02&1.008e-02&9.693e-03&1.139e-02&1.474e-02
\\
\noalign{\smallskip}
 \hline
\noalign{\smallskip}
Degree&6&7&8&9&10&11\\ [1ex]
$E_N$&
 2.023e-02&   2.892e-02&   4.261e-02&   6.429e-02&   9.886e-02 &  1.544e-01\\
\noalign{\smallskip}
 \hline
\noalign{\smallskip}
Degree&12&13&14&15&16&17\\ [1ex]
$E_N$&
2.444e-01 &  3.912e-01&   6.322e-01 &  1.030e+0&   1.692e+0&   2.796e+0
\\
\noalign{\smallskip}
 \hline
\noalign{\smallskip}
\end{tabular}
\end{table}
\end{remark}

\section{Conclusion}
\label{sec:conclude}
A dispersion analysis was presented 
for the energy-conserving DG method \cite{FuShu18}
for the one-wave wave equation.
Method with parameter $\alpha=1$
is shown to be superior to both the upwinding DG method and centered DG method in terms 
of dispersion error, with the leading term for the relative error $\rho_N$ of order 
$\Omega^{2N+3}$ for any polynomial degree $N$.
A judicious choice of the parameter $\alpha$ \eqref{alpha-opt} 
gives method (A*) which was shown to enjoy 
a leading term of order $\Omega^{2N+5}$ for the error $\rho_N$.

\bibliographystyle{siam}

\end{document}